\documentclass[a4paper,10pt]{amsart}
\usepackage[top=0.8in, bottom=0.8in, left=1.25in, right=1.25in]{geometry}
\usepackage{fancybox}
\usepackage{amsmath}
\usepackage{amssymb}
\usepackage{amsthm}
\usepackage{mathrsfs}
\usepackage{epstopdf}
\usepackage{epstopdf}
\usepackage{array} 
\usepackage{amscd}
\usepackage{color}
\usepackage{hyperref}
\usepackage{verbatim}

\usepackage[all]{xypic}
\usepackage{tikz}
\usepackage{tikz-cd}

\usepackage{amsfonts}
\usepackage{amssymb}
\usepackage{mathrsfs}

\DeclareFontFamily{U}{rsfs}{%
\skewchar\font127}
\DeclareFontShape{U}{rsfs}{m}{n}{%
<-6>rsfs5<6-8.5>rsfs7<8.5->rsfs10}{}
\DeclareSymbolFont{rsfs}{U}{rsfs}{m}{n}
\DeclareSymbolFontAlphabet
{\mathrsfs}{rsfs}
\DeclareRobustCommand*\rsfs{%
\@fontswitch\relax\mathrsfs}

\theoremstyle{plain}
\newtheorem{thm}{Theorem}[section]
\newtheorem{prop}[thm]{Proposition}
\newtheorem{lem}[thm]{Lemma}

\newtheorem{defi}[thm]{Definition}

\newtheorem{rmk}[thm]{Remark}

\newtheorem{prop-defi}[thm]{Proposition-Definition}
\newtheorem{thm-defi}[thm]{Theorem-Definition}
\newtheorem{lem-defi}[thm]{Lemma-Definition}

\newtheorem{conj}[thm]{Conjecture}

\newtheorem{example}[thm]{Example}
\newtheorem{set}[thm]{Setting}

\newcommand{\vir}{\mathrm{vir}}

\newdimen\argwidth
\def\db[#1\db]{
 \setbox0=\hbox{$#1$}\argwidth=\wd0
 \setbox0=\hbox{$\left[\box0\right]$}
  \advance\argwidth by -\wd0
 \left[\kern.3\argwidth\box0 \kern.3\argwidth\right]}

\newcommand{\cC}{\mathcal{C}}

\newcommand{\eE}{\mathcal{E}}
\newcommand{\fF}{\mathcal{F}}
\newcommand{\gG}{\mathcal{G}}
\newcommand{\hH}{\mathcal{H}}

\newcommand{\oO}{\mathcal{O}}
\newcommand{\pP}{\mathcal{P}}
\newcommand{\qQ}{\mathcal{Q}}

\newcommand{\dR}{\mathbf{R}}

\newcommand{\Hilb}{\mathop{\rm Hilb}\nolimits}

\newcommand{\Pic}{\mathop{\rm Pic}\nolimits}

\newcommand{\id}{\textrm{id}}

\newcommand{\ch}{\mathop{\rm ch}\nolimits}

\newcommand{\td}{\mathop{\rm td}\nolimits}
\newcommand{\Ext}{\mathop{\rm Ext}\nolimits}

\newcommand{\Coh}{\mathop{\rm Coh}\nolimits}

\newcommand{\ev}{\mathop{\rm ev}\nolimits}

\newcommand{\cneq}{\mathrel{\raise.095ex\hbox{:}\mkern-4.2mu=}}
\newcommand{\eqcn}{\mathrel{=\mkern-4.5mu\raise.095ex\hbox{:}}}

\newcommand{\ext}{\mathop{\rm ext}\nolimits}

\newcommand{\DT}{\mathop{\rm DT}\nolimits}
\newcommand{\PT}{\mathop{\rm PT}\nolimits}

\newcommand{\Sym}{\mathop{\rm Sym}\nolimits}

\newcommand{\End}{\mathop{\rm End}\nolimits}

\newcommand{\Ker}{\mathop{\rm Ker}\nolimits}

\newcommand{\RHom}{\mathop{\dR\mathrm{Hom}}\nolimits}

\newcommand{\BC}{{\mathbb{C}}}

\newcommand{\BF}{{\mathbb{F}}}

\newcommand{\BQ}{{\mathbb{Q}}}

\newcommand{\BZ}{{\mathbb{Z}}}

\newcommand{\CI}{{\mathcal I}}

\newcommand{\CL}{{\mathcal L}}

\newcommand{\CO}{{\mathcal O}}

\newcommand{\CZ}{{\mathcal Z}}

\newcommand{\fG}{{\mathfrak{G}}}

\newcommand{\pt}{{\mathsf{p}}}
\newcommand{\p}{{\mathbb{P}}}

\makeatletter
 
  \@addtoreset{equation}{section}
\makeatother

\title[{Stable pairs and GV type invariants on holomorphic symplectic  4-folds}]
{Stable pairs and Gopakumar-Vafa type invariants \\ on holomorphic symplectic 4-folds}
\date{}
\author{Yalong Cao}
\address{RIKEN Interdisciplinary Theoretical and Mathematical Sciences Program (iTHEMS), 2-1, Hirosawa, Wako-shi, Saitama, 351-0198, Japan}
\email{yalong.cao@riken.jp}
\author{Georg Oberdieck}
\address{University of Bonn, Institut f\"ur Mathematik}
\email{georgo@math.uni-bonn.de}
\author{Yukinobu Toda}
\address{Kavli Institute for the Physics and Mathematics of the Universe (WPI), The University of Tokyo Institutes for Advanced Study, The University of Tokyo, Kashiwa, Chiba 277-8583, Japan}
\email{yukinobu.toda@ipmu.jp}

\begin{document}
\maketitle
\begin{abstract}
As an analogy to Gopakumar-Vafa conjecture on Calabi-Yau 3-folds, Klemm-Pandharipande defined Gopakumar-Vafa type invariants of a Calabi-Yau 4-fold $X$ using Gromov-Witten theory. When $X$ is holomorphic symplectic, Gromov-Witten invariants 
vanish and one can consider the corresponding reduced theory. In a companion work, we propose a definition of Gopakumar-Vafa type invariants for such a reduced theory. In this paper, we give them a sheaf theoretic interpretation via moduli spaces of stable pairs. 
\end{abstract}

${}$  

${}$ 

\setcounter{tocdepth}{1}

\section{Introduction}

\subsection{Gopakumar-Vafa invariants}
A smooth complex projective 4-fold $X$ is {\em holomorphic symplectic} if
it is equipped with a non-degenerate holomorphic 2-form $\sigma\in H^0(X,\Omega^2_X)$. 
The ordinary Gromov-Witten invariants of $X$ always vanish for non-zero curve classes. Instead a reduced Gromov-Witten theory is
defined by Kiem-Li's cosection localization \cite{KiL}. 

Given cohomology classes $\gamma_i \in H^{\ast}(X,\mathbb{Z})$,
the (reduced) Gromov-Witten invariants of $X$ in 
a non-zero curve class 
$\beta \in H_2(X,\mathbb{Z})$ 
are defined by
\begin{align}\label{intro GWinv}
\mathrm{GW}_{g, \beta}(\gamma_1, \ldots, \gamma_l)
=\int_{[\overline{M}_{g, l}(X, \beta)]^{\rm{vir}}}
\prod_{i=1}^l \mathrm{ev}_i^{\ast}(\gamma_i),
\end{align}
where 
\begin{equation*}[\overline{M}_{g, l}(X, \beta)]^{\vir}\in A_{2-g+l}(\overline{M}_{g, l}(X, \beta)) \end{equation*}
is the (reduced) virtual class and  $\mathrm{ev}_i \colon \overline{M}_{g,l}(X, \beta)\to X$
is the evaluation map at the $i$-th marking. 
We refer to \cite{O1, O3, OSY} for some references on computations for \eqref{intro GWinv}.
Gromov-Witten invariants are in general rational numbers because the moduli space $\overline{M}_{g,l}(X, \beta)$ of stable maps is a Deligne-Mumford stack. 
It is an interesting question to find out integer-valued invariants
which underlie them.

In \cite{COT1}, we studied this question and defined \textit{genus 0 Gopakumar-Vafa invariants}
\begin{equation}\label{intro gv invs1}n_{0,\beta}(\gamma_1, \ldots, \gamma_l)\in \mathbb{Q} \end{equation}
for any non-zero curve class $\beta$ and \textit{genus 1 and 2 Gopakumar-Vafa invariants}
\begin{equation}\label{intro gv invs2}n_{1,\beta}(\gamma)\in \mathbb{Q}, \,\,\forall \,\, \gamma\in H^4(X,\mathbb{Z}); \quad n_{2,\beta} \in \mathbb{Q} \end{equation}
for any primitive curve class $\beta$ (i.e.~it is not a multiple of a non-zero curve class in $H_2(X,\BZ)$) from Gromov-Witten invariants \eqref{intro GWinv} (see \S \ref{sect on gv} for details). 
This may be compared with the previous works of Gopakumar and Vafa \cite{GV} on Calabi-Yau 3-folds, Klemm and Pandharipande \cite{KP} on Calabi-Yau 4-folds and Pandharipande and Zinger \cite{PZ} on Calabi-Yau 5-folds. 

In loc.~cit., we conjectured the integrality of \eqref{intro gv invs1}, \eqref{intro gv invs2} and provided substantial evidence for it. 
The aim of this paper is to give a sheaf theoretic interpretation of these Gopakumar-Vafa invariants using moduli spaces of stable pairs, in analogy with the discussion of \cite{CMT2, CT1} on ordinary Calabi-Yau 4-folds. 

\subsection{GV/Pairs correspondence}
Let $F$ be a one dimensional coherent sheaf on $X$ and $s\in H^0(F)$ be a section.
For an ample divisor $\omega$ on $X$, we denote the slope function by $\mu(F)=\chi(F)/(\omega \cdot [F])$.
The pair $(F,s)$ is called $Z_t$-\textit{stable} $($$t\in\mathbb{R}$$)$ if
\begin{enumerate}
\renewcommand{\labelenumi}{(\roman{enumi})}
\item for any subsheaf $0\neq F' \subseteq F$, we have 
$\mu(F')<t$,
\item for any
subsheaf $ F' \subsetneq F$ 
such that $s$ factors through $F'$, 
we have 
$\mu(F/F')>t$. 
\end{enumerate}
For a non-zero curve class $\beta \in H_2(X, \mathbb{Z})$ and $n\in \mathbb{Z}$, 
we denote by
\begin{align*}
P_n^t(X, \beta)
\end{align*}
the moduli space of 
$Z_t$-stable pairs 
$(F, s)$ with $([F], \chi(F))=(\beta, n)$. It has a wall-chamber structure and for a \textit{general} $t \in \mathbb{R}$ (i.e.~outside a finite subset of rational numbers in $\mathbb{R}$),  
it is a projective scheme. 

When $t<\frac{n}{\omega \cdot \beta}$, $P_n^t(X, \beta)$ is empty. The first nontrivial chamber appears when 
$t=\frac{n}{\omega \cdot \beta}+0^+$, which we call \textit{Joyce-Song (JS) chamber} (here $0^+$ denotes a sufficiently small positive number 
with respect to the fixed $\omega,\beta,n$). When $t\gg 1$, it recovers the moduli space of \textit{Pandharipande-Thomas (PT) stable pairs} \cite{PT} (Proposition \ref{prop:chambers}).

For general $t\in \mathbb{R}$, by Theorem \ref{existence of proj moduli space}, we can define its $\DT_4$ virtual class following \cite{BJ, OT} (see also \cite{CL1}).
However, by a cosection argument the virtual class vanishes, see \cite{KiP, Sav}. 
Using Kiem-Park's cosection localization \cite{KiP}, we have a (reduced) virtual class
\begin{align*}[P^t_n(X,\beta)]^{\vir}\in A_{n+1}(P^t_n(X,\beta),\mathbb{Q}), \end{align*}
depending on the choice of orientation \cite{CGJ, CL2}. More precisely, for each connected component of $P^t_n(X,\beta)$, there 
are two choices of orientation which affect the virtual class by a sign (component-wise).
To define its counting invariants, let    
\begin{align*}\tau: H^{m}(X,\mathbb{Z})\to H^{m-2}(P_n^t(X,\beta),\mathbb{Z}), \end{align*}
\begin{align*}\tau(\gamma):=\pi_{P\ast}\left(\pi_X^{\ast}\gamma \cup\ch_{3}(\mathbb{F})\right),
\end{align*}
where $\mathbb{I}=(\oO\to \mathbb{F})$ is the universal $Z_t$-stable pair and $\pi_P, \pi_X$ are projections from $P_n^t(X,\beta)\times X$ onto its factors. 
For $\gamma_i \in H^{m_i}(X, \mathbb{Z})$, the $Z_t$-\textit{stable pair invariants} are defined by 
\begin{align*}  
P_{n,\beta}^t(\gamma_1,\ldots,\gamma_l):=\int_{[P_n^t(X,\beta)]^{\rm{vir}}}\prod_{i=1}^l\tau(\gamma_i)\in\mathbb{Q}.
\end{align*}
When $n=-1$, we also write 
$$P_{-1,\beta}^t:=\int_{[P_{-1}^t(X,\beta)]^{\rm{vir}}}1. $$
Here is the main conjecture of this paper, which gives a sheaf theoretic interpretation of
all genus Gopakumar-Vafa invariants using $Z_t$-stable pair invariants. 
\begin{conj}\emph{(Conjecture \ref{conj on DT4/GV})}\label{intro conj on DT4/GV}
Fix $n\in\mathbb{Z}$, $\beta\in H_2(X,\mathbb{Z})$ and let $t>\frac{n}{\omega\cdot \beta}$ be generic. 
For certain choice of orientation, we have 
\begin{enumerate}
\item If $n\geqslant 2$, then 
\begin{align*} 
P_{n,\beta}^t(\gamma_1,\ldots,\gamma_l)=0.
\end{align*}
\item If $n=1$, then
\begin{align*} 
P_{1,\beta}^t(\gamma_1,\ldots,\gamma_l)=n_{0, \beta}(\gamma_1, \ldots, \gamma_l). \end{align*}
\item If $n=0$ and $\beta$ is primitive, then
\begin{align*} 
P_{0,\beta}^t(\gamma)=n_{1, \beta}(\gamma).
\end{align*}
\item If $n=-1$ and $\beta$ is primitive, then 
\begin{align*} 
P_{-1,\beta}^t=n_{2,\beta}.
\end{align*}
\end{enumerate}
\end{conj}
We verify this conjecture by a computation in an ideal geometry where curves deform in families of expected dimensions and 
have expected generic properties (see \S \ref{sect on heur}). Besides this, we study several examples and prove our conjecture in those cases.

\subsection{Verification of conjectures I: $K3\times K3$}
Let $X=S\times T$ be the product of two $K3$ surfaces. 
When the curve class $\beta \in H_2(S \times T, \BZ)$
is of non-trivial degree over both $S$ and $T$, then
one can construct two linearly independent cosections for moduli spaces of stable maps,
which imply that the (reduced) Gromov-Witten invariants of $X$ in this class vanish. Therefore we always restrict to consider curve classes of form
\begin{equation*}\beta\in H_2(S,\mathbb{Z})\subseteq H_2(X,\mathbb{Z}). \end{equation*} 

\begin{thm}\emph{(Theorem \ref{thm on g=0 conj on prod}, \ref{thm on g=1 conj on prod}, \ref{thm on P_-1}, Remark \ref{rmk on pri g=0})}
Let $X=S\times T$ be as above. Then Conjecture \ref{intro conj on DT4/GV}
holds for any primitive curve class $\beta\in H_2(S,\mathbb{Z})\subseteq H_2(X,\mathbb{Z})$. 
\end{thm}
In fact, by the global Torelli theorem (see e.g.~\cite{Ver, Huy}), primitive curve classes on $K3$ surfaces can be deformed to irreducible curve classes. 
By deformation invariance, we only need to deal with an irreducible curve class $\beta$, in which case we have an isomorphism (Proposition \ref{prop on smoothness}):
\begin{equation*}P^t_{n}(X,\beta)\cong P^t_{n}(S,\beta)\times T, \end{equation*}
and a forgetful map 
\begin{equation}\label{intro fort map}P^t_{n}(S,\beta)\to M_n(S,\beta), \end{equation}
where $M_n(S,\beta)$ is the coarse moduli space of one dimensional stable sheaves $F$ on $S$ with $[F]=\beta$, $\chi(F)=n$.
Both $P^t_{n}(S,\beta)$ and $M_n(S,\beta)$ are smooth schemes. We can then determine the $\DT_4$ virtual class of $P^t_{n}(X,\beta)$ (Theorem \ref{thm on vir clas}) and its pushforward (under the forgetful map) by the Thom-Porteous formula (Proposition \ref{deg loci}). This enables us to reduce the computation
of $Z_t$-stable pair invariants to certain tautological integrals on $M_n(S,\beta)$.
By Markman's framework of monodromy operators \cite{Markman}, we relate such integrals to certain tautological integrals
on Hilbert schemes of points on $S$ (see \S \ref{sect on trans} for details), 
which we explicitly determine using \cite{COT1} (see the proof of Theorem \ref{thm on g=1 conj on prod}, \ref{thm on P_-1} for details).

\subsection{Verification of conjectures II: $T^*\mathbb{P}^2$}
Let $H \in H^2(T^{\ast} \p^2)$ be the pullback of the hyperplane class and let us identify $H_2(T^{\ast} \p^2, \BZ) \equiv \BZ$ by taking the degree against $H$.

By explicitly describing the moduli spaces and virtual classes, we obtain: 
\begin{prop}\emph{(Proposition \ref{prop on tp2})}\label{intro prop on tp2}
For certain choice of orientation, we have 
$$P_{1,1}(H^2,H^2)=1, \quad P_{1,2}(H^2,H^2)=-1, \quad P_{1,3}(H^2,H^2)=0, $$
$$P_{0,1}(H^2)=P_{0,2}(H^2)=0, \quad P_{0,3}(H^2)=1, \quad P_{-1,1}=P_{-1,2}=P_{-1,3}=0.$$
Moreover, $P^t_{n}(X,d)$ is independent of the choice of $t>n/d$ in the listed cases above. 

In particular, for $X=T^*\mathbb{P}^2$, we have
\begin{itemize}
\item Conjecture \ref{intro conj on DT4/GV} (2) holds when $d\leqslant 3$. 
\item Conjecture \ref{intro conj on DT4/GV} (3), (4) hold. 
\end{itemize}
\end{prop}

\subsection{Verification of conjectures III: exceptional curves on $\Hilb^2(K3)$} 
Let $S$ be a $K3$ surface and $\Hilb^2(S)$ be the Hilbert scheme of two points on $S$. Consider the Hilbert-Chow map 
$$\pi: \Hilb^2(S)\to \Sym^2(S) $$
to the symmetric product of $S$. Let $D$ be the exceptional divisor fitting into Cartesian diagram: 
\begin{align*} \xymatrix{
D \ar[d]_{\pi}  \ar[r]^{i \quad \,\,\, }    & \Hilb^2(S) \ar[d]^{\pi}   \\
S    \ar[r]^{\Delta \quad \,\,\, }   & \Sym^2(S),  } \quad \quad
\end{align*}
where $\Delta$ is the diagonal embedding and $\pi: D\to S$ is a $\mathbb{P}^1$-bundle.
The following provides a verification of our (genus 0) conjecture for imprimitive curve classes.  
\begin{thm}\emph{(Theorem \ref{thm on hilbS})}\label{intro thm on hilbS}
In the JS chamber, Conjecture \ref{intro conj on DT4/GV} (1),~(2) hold for multiple fiber classes $\beta=r[\mathbb{P}^1]$ $($$r\geqslant 1$$)$ of $\pi$ as above. 
\end{thm} 
In fact, by the Jordan-H\"older filtration and a dimension counting, the JS pair invariants of $P_{n}^{\mathrm{JS}}(X,r[\mathbb{P}^1])$ are zero unless $n=r$ and 
in which case we have 
$$P_{n}^{\mathrm{JS}}(X,n[\mathbb{P}^1]) \cong \Hilb^n(S). $$
Then the proof makes use of the Chern class operator of tautological bundles by Lehn \cite{Lehn}. 

\subsection{Multiple fiber classes of elliptic fibrations} 
Let $p: S\rightarrow\mathbb{P}^{1}$ be an elliptic $K3$ surface and consider the elliptic fibration: 
$$\bar{p}:=p\times \id_T: X:=S\times T\to \mathbb{P}^{1}\times T=:Y, $$
where $T$ is a $K3$ surface.
Denote $f$ to be a generic fiber of $\bar{p}$ and $\pt\in H_0(T)$ be the point class.
The following gives a closed formula of $Z_t$-stable pair invariants for multiple fiber classes. 
\begin{thm}\emph{(Theorem \ref{thm2 on g=1 of multiple fiber})}\label{intro thm2 on g=1 of multiple fiber}
Let $t>0$. Then for certain choice of orientation, we have 
\begin{align*}
\sum_{r\geqslant 0}P^t_{0,r[f]}(\gamma)\,q^r=24\,\left(\int_{S \times \pt} \gamma\right)\cdot \sum_{m\geqslant 1}\sum_{n | m}n^2q^m. \end{align*}
\end{thm}
As for the proof, we note that there is an isomorphism 
$$\bar{p}^*: \Hilb^r(Y) \cong P^t_0(X,r[f]), \quad I_Z\mapsto \bar{p}^*I_Z, $$
under which the (reduced) virtual classes 
$$(-1)^{n+1}[\Hilb^r(Y)]^{\vir}=[P^t_0(X,r[f])]^{\vir}\in A_1(\Hilb^r(Y)) $$
can be identified for certain choice of orientation on the right hand side. 
Then we are left to evaluate an integral on $[\Hilb^r(Y)]^{\vir}$ which can be done via the degeneration method and a Behrend function argument \cite{B, OS}.
We refer to Theorem \ref{thm1 on g=1 of multiple fiber} for a similar result for trivial elliptic fibration $E\times E\times T\to E\times T$ 
and the proof therein for details. 

The formula in Theorem \ref{intro thm2 on g=1 of multiple fiber} seems to support our speculation of a GV/Pairs correspondence in genus 1 for imprimivite curve classes (see \S \ref{sect on impri} for details). 

\subsection{A conjectural virtual pushforward formula} 
Finally we remark that for a general holomorphic symplectic 4-fold $X$ and an irreducible curve class  $\beta\in H_2(X,\mathbb{Z})$, 
we have a forgetful map as in \eqref{intro fort map}:
\begin{equation*}P^t_{n}(X,\beta)\to M_n(X,\beta), \end{equation*}
where $M_n(S,\beta)$ is the coarse moduli space of one dimensional stable sheaves $F$ on $X$ with $[F]=\beta$, $\chi(F)=n$.
In Appendix \S \ref{sect on app}, we conjecture a virtual pushforward formula for this map (which we verify for the product of $K3$ surfaces, see Proposition \ref{prop on prod of k3 app}). Together with Conjecture \ref{intro conj on DT4/GV} (4), this formula implies a conjectural relation between genus 2 Gopakumar-Vafa invariants and certain descendent invariants on 
$M_1(X,\beta)$ (Proposition \ref{prop on appe}), which appears as \cite[Conj.~2.2 (iii)]{COT1}.

\subsection{Notation and convention}
In this paper, all varieties and schemes are defined over $\mathbb{C}$. 
For a morphism $\pi \colon X \to Y$ of schemes, 
and for $\fF, \gG \in \mathrm{D^{b}(Coh(\textit{X\,}))}$, we denote by 
$\dR \hH om_{\pi}(\fF, \gG)$ 
the functor $\dR \pi_{\ast} \dR \hH om_X(\fF, \gG)$. 

A class $\beta\in H_2(X,\mathbb{Z})$ is called \textit{effective} if there exists a non-empty curve $C \subset X$ with class $[C] = \beta$. An effective class $\beta$ is called \textit{irreducible} if it is not the sum of two effective classes, and it is called \textit{primitive} if it is not a positive integer multiple of an effective class.

A holomorphic-symplectic variety is a smooth projective variety
with a non-degenerate holomorphic two form $\sigma\in H^0(X,\Omega^2_X)$. 
A holomorphic-symplectic variety is irreducible \textit{hyperk\"ahler}
if $X$ is simply connected and $H^0(X, \Omega_X^2)$ is generated by a symplectic form.
A $K3$ surface is an (irreducible) hyperk\"ahler  variety of dimension $2$.


\subsection*{Acknowledgement}  
We thank Luca Battistella, Chen Jiang, Young-Hoon Kiem, Sergej Monavari, Rahul Pandharipande and Hyeonjun Park for helpful discussions.
We are grateful to the referee for a careful reading of our paper and providing many helpful comments which improved the
exposition of this paper. 

Y. C. is partially supported by RIKEN Interdisciplinary Theoretical and Mathematical Sciences
Program (iTHEMS), World Premier International Research Center Initiative (WPI), MEXT, Japan, 
JSPS KAKENHI Grant Number JP19K23397 and Royal Society Newton International Fellowships Alumni 2020 and 2021. 
G.O. is partially supported by Deutsche Forschungsgemeinschaft (DFG) - OB 512/1-1. 
Y. T. is partially supported by World Premier International Research Center Initiative (WPI initiative), MEXT, Japan, and
Grant-in Aid for Scientific Research grant (No. 19H01779) from MEXT, Japan.


\section{Definitions and conjectures }

\subsection{Gopakumar-Vafa invariants}\label{sect on gv}

Let $X$ be a holomorphic symplectic 4-fold and $\overline{M}_{g, l}(X, \beta)$
be the moduli stack of genus $g$, $l$-pointed stable maps
to $X$ with non-zero curve class $\beta$. Its virtual class \cite{BF, LT} vanishes due to a trivial factor in the obstruction sheaf.
By Kiem-Li's theory of cosection localization \cite{KiL}, one can define a (reduced) virtual class\,\footnote{The virtual class mentioned in this paper
is always assumed to be the reduced one.}
$$[\overline{M}_{g, l}(X, \beta)]^{\vir}\in A_{2-g+l}(\overline{M}_{g, l}(X, \beta)). $$
For integral classes
\begin{align}\label{gamma}
\gamma_i \in H^{m_i}(X, \mathbb{Z}), \
1\leqslant i\leqslant l,
\end{align}
the (primary) Gromov-Witten invariant is defined by
\begin{align}\label{GWinv}
\mathrm{GW}_{g, \beta}(\gamma_1, \ldots, \gamma_l)
=\int_{[\overline{M}_{g, l}(X, \beta)]^{\rm{vir}}}
\prod_{i=1}^l \mathrm{ev}_i^{\ast}(\gamma_i)\in \mathbb{Q},
\end{align}
where $\mathrm{ev}_i \colon \overline{M}_{g,l}(X, \beta)\to X$
is the $i$-th evaluation map.

When $g=0$, the
virtual dimension of $\overline{M}_{0, l}(X, \beta)$
is $l+2$, and (\ref{GWinv})
is zero unless
\begin{align}\label{sum1}
\sum_{i=1}^{l}(m_i-2)=4.
\end{align}
Similar to the case of Calabi-Yau 4-folds and 5-folds \cite{KP, PZ}, we make the following definition: 
\begin{defi}\emph{(\cite[Def.~1.5]{COT1})}\label{def of g=0 GV inv}
For any $\gamma_1, \ldots, \gamma_l \in H^{\ast}(X,\BZ)$, 
we define the genus $0$ Gopakumar-Vafa invariant $n_{0, \beta}(\gamma_1, \ldots, \gamma_l) \in \BQ$ recursively by the multiple cover formula: 
$$\mathrm{GW}_{0, \beta}(\gamma_1, \ldots, \gamma_l)=\sum_{\begin{subarray}{c}k\geqslant 1, k|\beta  \end{subarray}}k^{n-3}\, n_{0, \beta/k}(\gamma_1, \ldots, \gamma_l). $$
\end{defi}
When $g=1$, the virtual dimension of
$\overline{M}_{1, l}(X, \beta)$ is $l+1$, and (\ref{GWinv})
is zero unless
\begin{align}\label{sum2}
\sum_{i=1}^{l}(m_i-2)=2.
\end{align}
In this paper, we concentrate on the case when $l=1$ and $m_1=4$. 
Because curves in imprimitive curve classes are very difficult to control,
we restrict hereby to the case of a primitive curve class.
\begin{defi}\emph{(\cite[Def.~1.6]{COT1})}\label{def of g=1 GV inv}
Assume that $\beta \in H_2(X,\BZ)$ is primitive. For any $\gamma\in H^4(X, \mathbb{Z})$, we define the genus 1 Gopakumar-Vafa invariant $n_{1, \beta}(\gamma)\in \mathbb{Q}$ by
$$\mathrm{GW}_{1, \beta}(\gamma)=n_{1,\beta}(\gamma) - \frac{1}{24} \mathrm{GW}_{0,\beta}(\gamma,c_2(X)), $$
where $c_2(X)$ is the second Chern class of $T_X$. 
\end{defi}
When $g=2$, the virtual dimension of
$\overline{M}_{2, 0}(X, \beta)$ is zero, so we can consider (\ref{GWinv}) without insertions:
\begin{align*}
\mathrm{GW}_{2, \beta}:=\int_{[\overline{M}_{2, 0}(X, \beta)]^{\rm{vir}}}1\in \mathbb{Q}.
\end{align*}
\begin{defi}\emph{(\cite[Def.~1.7]{COT1})}\label{def of g=2 GV inv}
Assume that $\beta \in H_2(X,\BZ)$ is primitive. We define the genus $2$ Gopakumar-Vafa invariant $n_{2,\beta}\in \mathbb{Q}$ by
\[\mathrm{GW}_{2, \beta}=n_{2,\beta}
- \frac{1}{24} n_{1,\beta}(c_2(X))
+ \frac{1}{2 \cdot 24^2} \mathrm{GW}_{0, \beta}(c_2(X),c_2(X))
+ \frac{1}{24} N_{\mathrm{nodal},\beta}. \]
Here $n_{1,\beta}(-)$ is given in Definition \ref{def of g=1 GV inv} and $N_{\mathrm{nodal},\beta}\in \mathbb{Q}$ is the virtual count of rational nodal curves \cite{NO} 
as defined by 
\begin{equation} \label{Nnodal}
N_{\mathrm{nodal},\beta}:=
\frac{1}{2}\left[
\int_{[\overline{M}_{0,2}(X,\beta)]^{\vir}} (\ev_1 \times \ev_2)^{\ast}(\Delta_X) - \int_{[ \overline{M}_{0,1}(X,\beta) ]^{\vir}} \frac{\ev_1^{\ast}(c(X))}{1-\psi_1}
\right], 
\end{equation}
where 
\begin{itemize}
\item $\Delta_X \in H^8(X \times X)$ is the class of the diagonal, and
\item $c(X) = 1 + c_2(X) + c_4(X)$ is the total Chern class of $T_X$.
\end{itemize}
 \end{defi}

\subsection{$Z_t$-stable pair invariants}

Let $\omega$ be an ample divisor on $X$ and $t\in\mathbb{R}$, we recall the following notion of $Z_t$-stable pairs.
\begin{defi}\label{def Zt sta}\emph{(\cite[Lem~1.7]{CT1})}
Let $F$ be a one dimensional coherent sheaf and $s: \oO_X\to F$ be a section. For an ample divisor $\omega$, we denote the slope function
by $\mu(F)=\chi(F)/(\omega \cdot [F])$.

We say $(F,s)$ is a $Z_t$-(semi)stable pair $($$t\in\mathbb{R}$$)$ if 
\begin{enumerate}
\renewcommand{\labelenumi}{(\roman{enumi})}
\item for any subsheaf $0\neq F' \subseteq F$, we have 
$\mu(F')<(\leqslant)t$,
\item for any
subsheaf $ F' \subsetneq F$ 
such that $s$ factors through $F'$, 
we have 
$\mu(F/F')>(\geqslant)t$. 
\end{enumerate}
\end{defi}
There are two distinguished stability conditions appearing as 
special cases of $Z_t$-stability. 
\begin{defi}\label{defi:PTJSpair}\emph{(\cite{PT}, \cite[Def.~1.10]{CT1})} 

(i) A pair $(F,s)$ is a PT stable pair if
$F$ is a pure one dimensional sheaf and $s$ is surjective in dimension one. 

(ii) A pair $(F,s)$ is a JS stable pair if $s$ is a non-zero morphism, $F$ is $\mu$-semistable and 
for any subsheaf $0\neq F' \subsetneq F$ such that $s$ factors through 
$F'$ we have $\mu(F')<\mu(F)$. 
\end{defi}
\begin{prop}\label{prop:chambers}\emph{(\cite[Prop.~1.11]{CT1})} 
For a pair $(F,s)$ with $[F]=\beta$ and $\chi(F)=n$, its

(i) $Z_t$-stability with $t\to \infty$ is exactly PT stability, 

(ii) $Z_t$-stability with $t=\frac{n}{\omega\cdot \beta}+0^+$ is exactly JS stability. 
\end{prop}
For $\beta \in H_2(X, \mathbb{Z})$ and $n\in \mathbb{Z}$, we denote by
$$P^t_n(X, \beta)\quad (\mathrm{resp}.\,\, \pP^t_n(X, \beta)) $$
the moduli stack of $Z_t$-stable (resp.~$Z_t$-semistable) pairs $(F,s)$ with $[F]=\beta$ and $\chi(F)=n$.

By Proposition \ref{prop:chambers}, there are two disinguished moduli spaces, 
PT moduli spaces and JS moduli spaces, 
by specializing $t\to \infty$ and $t=\frac{n}{\omega\cdot \beta}+0^+$ respectively:
\begin{align*}
P_n(X, \beta) \cneq P_n^{t\to \infty}(X, \beta), \quad
P_n^{\mathrm{JS}}(X, \beta) \cneq 
P_n^{t=\frac{n}{\omega\cdot \beta}+0^+}(X, \beta). 
\end{align*}
By a GIT construction,  
$P^t_n(X, \beta)$ is a quasi-projective scheme, and $\pP^t_n(X, \beta)$ admits a good moduli space
\begin{align*}
\pP^t_n(X, \beta) \to \overline{P}_n^t(X, \beta),
\end{align*}
where $\overline{P}_n^t(X, \beta)$ is a projective 
scheme which parametrizes $Z_t$-polystable objects.
The following result shows that moduli stacks of $Z_t$-stable pairs are indeed open substacks of moduli stacks of objects in the derived categories of coherent sheaves.
\begin{thm}\label{existence of proj moduli space}\emph{(\cite[Thm.~0.1]{CT1})} 
$P^t_n(X, \beta)$ admits an open immersion 
$$P^t_n(X, \beta)\to \mathcal{M}_0, \quad (F,s)\mapsto (\oO_X\stackrel{s}{\to} F) $$
to the moduli stack $\mathcal{M}_0 $ of $E\in D^b\Coh (X)$ with $\Ext^{<0}(E,E)=0$ and $\det(E)\cong \oO_X$.
\end{thm}
Therefore for a general choice of $t$ (i.e.~outside a finite subset of rational numbers in $\mathbb{R}$), $P^t_n(X, \beta)$ is a projective scheme which can
be given a $(-2)$-shifted symplectic derived scheme structure \cite{PTVV} and has a virtual class \cite{BJ, OT} (see also \cite{CL1}). 

Parallel to GW theory, the virtual class of $P_n^t(X,\beta)$ vanishes \cite{KiP, Sav}. 
One can define 
a reduced virtual class due to Kiem-Park \cite[Def.~8.7, Lem.~9.4]{KiP}: 
\begin{align}\label{red vir class}[P_n^t(X,\beta)]^{\vir}\in A_{n+1}(P_n^t(X,\beta),\mathbb{Q}), \end{align}
depending on the choice of orientation \cite{CGJ, CL2}. 
To define its counting invariants, let    
\begin{align}\label{equ on pri ins}\tau: H^{m}(X,\mathbb{Z})\to H^{m-2}(P_n^t(X,\beta),\mathbb{Z}), \end{align}
\begin{align*}\tau(\gamma):=\pi_{P\ast}\left(\pi_X^{\ast}\gamma \cup\ch_{3}(\mathbb{F})\right),
\end{align*}
where $\mathbb{I}=(\oO\to \mathbb{F})$ is the universal $Z_t$-stable pair and $\pi_P, \pi_X$ are projections from $P_n^t(X,\beta)\times X$ onto its factors. 
\begin{defi}\label{def DT4 inv}
Let $t\in \mathbb{R}$ be generic and $\gamma_i \in H^{m_i}(X, \mathbb{Z})$ $(1\leqslant i\leqslant l)$. The $Z_t$-stable pair invariants are    
\begin{align*}  
P_{n,\beta}^t(\gamma_1,\ldots,\gamma_l):=\int_{[P_n^t(X,\beta)]^{\rm{vir}}}\prod_{i=1}^l\tau(\gamma_i)\in\mathbb{Q}.
\end{align*}
When $n=-1$, we write 
$$P_{-1,\beta}^t:=\int_{[P_{-1}^t(X,\beta)]^{\rm{vir}}}1. $$
In PT and JS stabilities, we also write 
$$P_{n,\beta}(\gamma_1,\ldots,\gamma_l):=P_{n,\beta}^{t\to \infty}(\gamma_1,\ldots,\gamma_l), \,\,\, 
P^{\mathrm{JS}}_{n,\beta}(\gamma_1,\ldots,\gamma_l):=P_{n,\beta}^{t=\frac{n}{\omega\cdot \beta}+0^+
}(\gamma_1,\ldots,\gamma_l). $$
\end{defi}
\begin{rmk}
By Definition \ref{def Zt sta} and a dimension counting, $Z_t$-stable pair invariants are non-zero only if  
both of the following conditions hold: 
\begin{align*} 
t>\frac{n}{\omega\cdot \beta}, \quad \sum_{i=1}^{l}(m_i-2)=2n+2.
\end{align*}
\end{rmk}

In \cite{CMT2, CT1}, similar invariants are used to give sheaf theoretic interpretations of Gopakumar-Vafa type invariants for ordinary Calabi-Yau 4-folds \cite{KP}. 
Below, we give a parallel proposal for holomorphic symplectic 4-folds using Definition \ref{def DT4 inv}.

\subsection{Conjecture}
We state the main conjecture of this paper. 
\begin{conj}\label{conj on DT4/GV}
Let $X$ be a holomorphic symplectic 4-fold with an ample divisor $\omega$.
Fix $n\in\mathbb{Z}$ and $\beta\in H_2(X,\mathbb{Z})$ and let $t>\frac{n}{\omega\cdot \beta}$ be generic. 
For certain choice of orientation, we have 
\begin{enumerate}
\item If $n\geqslant 2$, then 
\begin{align*} 
P_{n,\beta}^t(\gamma_1,\ldots,\gamma_l)=0.
\end{align*}
\item If $n=1$, then
\begin{align*} 
P_{1,\beta}^t(\gamma_1,\ldots,\gamma_l)=n_{0, \beta}(\gamma_1, \ldots, \gamma_l) \in \mathbb{Z}. \end{align*}
\item If $n=0$ and $\beta$ is primitive, then
\begin{align*} 
P_{0,\beta}^t(\gamma)=n_{1, \beta}(\gamma) \in \mathbb{Z}.
\end{align*}
\item If $n=-1$ and $\beta$ is primitive, then 
\begin{align*} 
P_{-1,\beta}^t=n_{2,\beta} \in \mathbb{Z}.
\end{align*}
\end{enumerate}
\end{conj}
\begin{rmk}
By the global Torelli theorem \cite{Ver, Huy}, primitive curve classes on irreducible hyperk\"ahler varieties can be deformed to irreducible curve classes. Therefore $Z_t$-stable pair invariants 
are independent of the choice of $t>\frac{n}{\omega\cdot \beta}$ for such cases by \cite[Prop.~1.12]{CT1}.
\end{rmk}
\begin{rmk}
Our conjecture implies that there is no nontrivial wall-crossing for $Z_t$-stable pairs invariants when $t>\frac{n}{\omega\cdot \beta}$, contrary to the 
ordinary $\mathrm{CY_4}$ case \cite{CT1, CT3, CT4}.
\end{rmk}
\begin{rmk}
Similarly to \cite[Conj.~0.3]{CK2}, we may use counting invariants on Hilbert schemes $I_n(X,\beta)$ of curves to give 
a sheaf theoretic interpretation of Gopakumar-Vafa invariants in which case zero dimensional subschemes \cite{CK1} (conjecturally) will not contribute, i.e. $``\DT=\PT"$. 
It is curious whether one can do a $K$-theoretic refinement as \cite{CKM1}. 
\end{rmk}

\subsection{Heuristic argument}\label{sect on heur}
In this section, we verify Conjecture \ref{conj on DT4/GV} using heuristic argument in an ideal geometry (ref.~\cite[\S 1.4, \S 1.5]{COT1}).
To be specific, as the virtual dimension of $\overline{M}_{g,0}(X,\beta)$ is $2-g$, we assume that:
\begin{quote}
Any genus $g$ curve moves in a smooth compact $(2-g)$-dimensional family.
\end{quote}
In particular, there are no curves of genus $g \geqslant 3$.
Unfortunately, complicated phenomena still arise even in the ideal case, for example, one can have two (resp.~one) dimensional 
families of reducible rational (resp.~elliptic) curves, and any member of a rational curve family is expected to intersect nontrivially with
some member in the same family (see \cite[\S 1.4]{COT1} for details). 

However, things will be simplified if we make the following additional assumptions:
\begin{itemize}
\item $X$ is irreducible hyperk\"ahler,
\item the effective curve class $\beta \in H_2(X,\BZ)$ is primitive,
\end{itemize}
By the global Torelli for (irreducible) hyperk\"ahler varieties \cite{Ver, Huy},
the pair $(X,\beta)$ is deformation equivalent (through a deformation with keeps $\beta$ of Hodge type)
to a pair $(X', \beta')$, where $\beta' \in H_2(X,\BZ)$ is irreducible, so we may without loss of generality assume:
\begin{itemize}
\item the effective curve class $\beta \in H_2(X,\BZ)$ is irreducible.
\end{itemize}
Under these assumptions, our ideal geometry of curves simplifies to the following form:
\begin{enumerate}
\item
The rational curves in $X$ of class $\beta$
move in a proper 2-dimensional smooth family of embedded irreducible rational curves. Except for a finite number of rational nodal curves, the rational curves are smooth, with normal bundle $\CO_{\p^1} \oplus \CO_{\p^1} \oplus \mathcal{O}_{\mathbb{P}^{1}}(-2)$. 
\item
The arithmetic genus $1$ curves in $X$ of class $\beta$ move in a proper 1-dimensional smooth family of embedded irreducible genus 1 curves. Except for a finite number of rational nodal curves, the genus one curves are smooth elliptic curves with normal bundle $L\oplus L^{-1}\oplus \oO$, where $L$ is a generic degree zero line bundle.
\item
All genus two curves are smooth and rigid.
\item
There are no curves of genus $g\geqslant   3$.
\end{enumerate}
We need to compute $Z_t$-stable pair invariants in this ideal setting. 
The key heuristic we use is that only $Z_t$-stable pairs with \textit{connected support} will `contribute' to our invariants. 

The observation is that for a $Z_t$-stable pair $I=(\oO_X\to F)$ such 
that $F$ 
is supported on a disconnected curve $C=C_1\sqcup C_2$, we may write 
$F=F_1 \oplus F_2$
where $F_i$ is supported on $C_i$ ($i=1,2$).
We set 
$$I_1=(\oO_X\to F_1), \,\, I_2=(\oO_X\to F_2).$$
 Then the obstruction space 
satisfies 
\begin{align}\label{Ext2:plus}\Ext^2(I,I)_0=\Ext^2(I_1,I_1)_0\oplus \Ext^2(I_2,I_2)_0. 
\end{align}
Indeed there is a distinguished triangle 
(see the argument of~\cite[(2.13)]{CMT2}):
\begin{align*}
    \dR \mathcal{H}om(I, F) \to \dR \mathcal{H}om(I, I)_0[1] \to \dR \mathcal{H}om(F, \mathcal{O}_X)[2], 
\end{align*}
and we have 
\begin{align}\notag
    \dR \mathcal{H}om(I, F)\cong
    \dR \mathcal{H}om(I, F_1) \oplus  \dR \mathcal{H}om(I, F_2)\cong\dR \mathcal{H}om(I_1, F_1) \oplus  \dR \mathcal{H}om(I_2, F_2), 
\end{align}
where the second isomorphism follows since
$I$ is isomorphic to $I_i$ near the support of $F_i$. 
Combining with 
$$\dR \mathcal{H}om(F, \mathcal{O}_X)=
\dR \mathcal{H}om(F_1, \mathcal{O}_X)
\oplus \dR \mathcal{H}om(F_2, \mathcal{O}_X),$$ 
we obtain  
\begin{align*}
   \dR \mathcal{H}om(I, I)_0[1]
   \cong \dR \mathcal{H}om(I_1, I_1)_0[1]
   \oplus \dR \mathcal{H}om(I_2, I_2)_0[1]. 
\end{align*}
Hence (\ref{Ext2:plus}) holds. 

Therefore the surjective isotropic cosections (see~\cite[Lem.~9.4]{KiP}) of obstruction spaces in the RHS of (\ref{Ext2:plus}) give rise to a (mutually orthogonal) two dimensional isotropic cosection in the LHS. Heuristically speaking, such $Z_t$-stable pairs will not `contribute' to the reduced virtual class as the reduced obstruction space still have a surjective isotropic cosection.

By Definition \ref{def DT4 inv} and above discussion, $Z_t$-stable pair invariants
\begin{align*}  
P_{n,\beta}^t(\gamma_1,\ldots,\gamma_l)=\int_{[P_n^t(X,\beta)]^{\rm{vir}}}\prod_{i=1}^l\tau(\gamma_i)
\end{align*}
count $Z_t$-stable pairs whose support are connected and incident to cycles dual to $\gamma_1,\ldots, \gamma_l$. 
Say such an incident $Z_t$-stable pair is supported on a $(2-g)$-dimensional family:
$$p:\cC^g_\beta\to S^g_\beta $$ 
of genus $g$ curves ($g=0,1,2$), where $\cC^g_\beta$ is the total space of this family.
Each cycle $\gamma_i$ will cut down real dimension of $S^g_\beta$ by $\deg(\gamma_i)-2$. As we have 
$$\sum_{i=1}^{l}(\deg(\gamma_i)-2)=2n+2, $$
so all insertions in total cut down real dimension of $S^g_\beta$ by $2n+2$. 

${}$ \\
\textbf{The case $n\geqslant 1$}. When $n\geqslant 2$, the dimension cut down by insertions is bigger than the largest possible dimension of $S^g_\beta$, so there can not be such incident stable pairs and 
$$P_{n\geqslant 2,\beta}^t(\gamma_1,\ldots,\gamma_l)=0. $$ 
This confirms Conjecture \ref{def DT4 inv} (1). 

When $n=1$, insertions cut down real dimension of $S^g_\beta$ by $4$, so any incident $Z_t$-stable pair $I=(\oO_X\to F)$ can only be supported on genus 0 family. 
As in \cite[\S4.1]{CT1}, by Harder-Narasimhan and Jordan-H\"older filtration, we know  
$$F\cong \oO_C, $$
for some rational curve $C$ in $S^0_{\beta}$. Therefore incident $Z_t$-stable pairs (with $\chi=1$) are in one to one correspondence with 
intersection points of $\cC^0_{\beta}$ with cycles dual to $\gamma_1,\ldots,\gamma_l$ and 
\begin{align*}  
P_{1,\beta}^t(\gamma_1,\ldots,\gamma_l)=\int_{S^0_\beta}\prod_{i=1}^lp_*(f^*\gamma_i),  \end{align*}
where $f: \mathcal{C}^0_{\beta}\to X$ is the evaluation map.
Therefore Conjecture \ref{def DT4 inv} (2) is confirmed in this ideal case as both sides are (virtually) enumerating rational curves of class 
$\beta$ incident to cycles dual to $\gamma_1,\ldots,\gamma_l$.

${}$ \\
\textbf{The case $n=0$}. Since $Z_t$-stable pairs $I=(\oO_X\to F)$ supported on genus $0$ curves satisfy $\chi(F)>0$ and  
a 4-cycle $\gamma\in H^4(X)$ misses genus 2 curves in general position, so when $[F]=\beta$ is irreducible, the pair 
must be scheme theoretically supported on an elliptic curve $C$ and 
$$I=(\oO_X\twoheadrightarrow \oO_C\stackrel{s}{\to} L), $$
where $L$ is a line bundle on $C$ with $\chi(C,L)=0$. By $Z_t$-stability, $s$ is non-trivial, so $s$ must 
be an isomorphism by the stability of line bundles. Therefore incident $Z_t$-stable pairs (with $\chi=0$) are in one to one correspondence with 
intersection points of 4-cycle $\gamma$ with genus 1 curve family $\cC^1_\beta$ of class $\beta$ and 
\begin{align*}  
P_{0,\beta}^t(\gamma)=\int_{\cC^1_\beta}f^*\gamma,  \end{align*}
where $f: \mathcal{C}^1_{\beta}\to X$ is the evaluation map. 
Therefore Conjecture \ref{def DT4 inv} (3) is confirmed in this ideal setting as both sides are (virtually) enumerating elliptic curves of class 
$\beta$ incident to $\gamma$.

${}$ \\
\textbf{The case $n=-1$}. Any $Z_t$-stable pair $I=(\oO_X\to F)$ with irreducible curve class $[F]=\beta$ 
is scheme theoretically supported on a smooth rigid genus 2 curve $C$: 
$$I=(\oO_X\twoheadrightarrow \oO_C\stackrel{s}{\to} L), $$
where $L$ is a line bundle on $C$ with $\chi(C,L)=-1$. As above, by $Z_t$-stability, $s$ must 
be an isomorphism. Hence $P_{-1}^t(X,\beta)$ is identified with 
the set of all rigid genus 2 curves in class $\beta$ in the ideal geometry, whose count gives exactly genus 2 Gopakumar-Vafa invariant $n_{2,\beta}$. 
Therefore Conjecture \ref{def DT4 inv} (4) is confirmed in the ideal setting.

\subsection{Speculations for general curve classes}\label{sect on impri}
For a smooth projective Calabi-Yau 4-fold $X$ and $\gamma\in H^4(X,\mathbb{Z})$, 
we have genus $0$, $1$ Gopakumar-Vafa type invariants $n_{0,\beta}(\gamma), n_{1,\beta}\in\mathbb{Q}$ defined 
by  Klemm and Pandharipande
from Gromov-Witten theory \cite{KP} and stable pair invariants $P_{n,\beta}(\gamma)\in \mathbb{Z}$ \cite{CMT2} (here $P_{n,\beta}(\gamma)$ is a shorthand for 
$P_{n,\beta}(\gamma,\ldots,\gamma)$). 
They are related by the following conjectural formula \cite[\S 1.7]{CMT2}: 
\begin{align}\label{equ on pt/gv on cy4}
\sum_{n,\beta}\frac{P_{n,\beta}(\gamma)}{n!}y^n q^{\beta}
=\prod_{\beta>0}\Big(\exp(yq^{\beta})^{n_{0,\beta}(\gamma)}\cdot 
M(q^{\beta})^{n_{1,\beta}}\Big), 
\end{align}
where $M(q)=\prod_{k\geqslant 1}(1-q^{k})^{-k}$ is the MacMahon function. 

By taking logarithmic differentiation with respect to $y$, we obtain 
\begin{align*}
y\frac{d}{dy}\log\left(\sum_{n,\beta}\frac{P_{n,\beta}(\gamma)}{n!}y^n q^{\beta}\right)=\sum_{\beta>0}y\frac{d}{dy}n_{0,\beta}(\gamma)yq^\beta =\sum_{\beta>0}n_{0,\beta}(\gamma)yq^\beta.
\end{align*}
If we view it as an equality for corresponding reduced invariants on holomorphic symplectic 4-folds, 
it surprisingly recovers Conjecture \ref{conj on DT4/GV} (i),~(ii) (i.e. the genus zero part). 

We do similar manipulations for genus one invariants. Note that $y^0q^\beta$ parts of \eqref{equ on pt/gv on cy4} are 
$$\sum_{\beta}P_{0,\beta}q^\beta=\prod_{\beta>0}M(q^{\beta})^{n_{1,\beta}}. $$
This equality is written down by a computation in the ``$\mathrm{CY_4}$ ideal geometry" (ref.~\cite[\S 2.5]{CMT2}), where rational curves contribute zero and each super-rigid elliptic curve 
(on an ideal $\mathrm{CY_4}$)
in class $\beta$ contributes by $M(q^\beta)$ (ref.~\cite[Thm.~5.10]{CMT2}). 
Taking logarithmic differentiation with respect to $q$:
\begin{align*}q\frac{d}{dq}\log\left(M(q) \right)=\sum_{d\geqslant 1}q^d\sum_{i\geqslant1,i|d}i^2.  \end{align*}
We then wonder whether in the holomorphic symplectic 4-folds setting, each ideal elliptic curve family in class $\beta$ contributes to $P_{0,d\beta}(\gamma)$ by 
$$\sum_{i\geqslant1,i|d}i^2. $$
Summing over all elliptic curve families, this would imply
\begin{align}\label{general pt/gv on hk4}P_{0,\beta}(\gamma)=\sum_{d\geqslant1,d|\beta}n_{1,\beta/d}(\gamma) \sum_{i\geqslant1,i|d}i^2.  \end{align}
It is quite curious whether the above formula gives the correct PT/GV correspondence. For multiple fiber classes 
of elliptic fibrations, our computations show the formula seems correct (see Theorem \ref{thm1 on g=1 of multiple fiber}, \ref{thm2 on g=1 of multiple fiber}, Remark \ref{rmk on impr}). As for $P_{-1,\beta}$ and genus 2 Gopakumar-Vafa invariants, we haven't found analogous 
formula for general curve classes.

\section{Product of $K3$ surfaces}
In this section, we consider the product of two $K3$ surfaces: 
$$X=S\times T, \quad \mathrm{with} \,\, \beta\in H_2(S,\mathbb{Z})\subseteq H_2(X,\mathbb{Z}). $$ 
As observed in \cite[\S 5]{COT1}, this contains all interesting curve classes on $X$ because if $\beta \in H_2(X, \BZ)$
is of non-trivial degree over both $S$ and $T$, one can construct two linearly independent cosections,
which imply that reduced Gromov-Witten invariants of $X$ in this class vanish.

\subsection{Gopakumar-Vafa invariants}
Recall Gopakumar-Vafa invariants specified in Definitions \ref{def of g=0 GV inv}, \ref{def of g=1 GV inv}, \ref{def of g=2 GV inv}. 
They are computed in \cite[Prop.~5.1]{COT1} as follows: 
write $\gamma,\gamma'\in H^{4}(X)$ as 
\begin{align*}
    \gamma&=A_1\cdot 1\otimes \pt+D_1\otimes D_2+A_2\cdot \pt\otimes 1, \\
\gamma'&=A'_1\cdot 1\otimes \pt+D'_1\otimes D'_2+A'_2\cdot \pt\otimes 1, 
\end{align*}
based on K\"unneth decomposition:
$$H^{4}(X)\cong (H^0(S)\otimes H^4(T))\oplus (H^2(S)\otimes H^2(T))\oplus (H^4(S)\otimes H^0(T)). $$
Fix also a curve class
$$\alpha=\theta_1\otimes \pt+\pt\otimes \theta_2\in H^6(X) \cong (H^2(S)\otimes H^4(T))\oplus (H^4(S)\otimes H^2(T)).$$
\begin{prop}\emph{(\cite[Prop.~5.1]{COT1})}\label{prop on gw on prod}
For $\beta\in H_2(S,\mathbb{Z})\subseteq H_2(X,\mathbb{Z})$, we have 
\begin{align*}
n_{0,\beta}(\gamma, \gamma') &=(D_1\cdot\beta)\cdot (D_1'\cdot\beta)\cdot\int_T(D_2\cdot D_2')\cdot N_{0}\left(\frac{\beta^2}{2}\right), \\
n_{0,\beta}(\alpha)&=(\theta_1\cdot \beta)\,N_{0}\left(\frac{\beta^2}{2}\right). 
\end{align*}
If $\beta$ is primitive, we have
\begin{align*}
n_{1, \beta}(\gamma)= 24 A_2\,  N_1\left(\frac{\beta^2}{2}\right), \quad n_{2,\beta}= N_2\left( \frac{\beta^2}{2} \right), 
\end{align*}
where
\begin{align}\sum_{l\in\mathbb{Z}}N_{0}(l)\, q^l&=\frac{1}{q} \prod_{n\geqslant   1}\frac{1}{(1-q^n)^{24}}, \label{equ on N0} \\
\sum_{l \in \BZ} N_{1}(l)\,q^l &=\left(\frac{1}{q} \prod_{n\geqslant   1}\frac{1}{(1-q^n)^{24}}\right)\left(q \frac{d}{dq}G_2(q)\right), \label{equ on N1} \\
\sum_{l\in\mathbb{Z}}N_{2}(l)\, q^l&=\left(\frac{1}{q} \prod_{n\geqslant   1}\frac{1}{(1-q^n)^{24}}\right) \left( 24 q \frac{d}{dq} G_2 - 24 G_2 - 1 \right),  \label{equ on N2}\end{align}
with Eisenstein series: 
$$G_2(q) = -\frac{1}{24} + \sum_{n \geqslant   1} \sum_{d|n} d q^n. $$
\end{prop}

\subsection{Moduli spaces of $Z_t$-stable pairs}
For a point $t\in T$, let $i_t \colon S\to S\times \{t\}\hookrightarrow X$ be the inclusion. 
Consider the pushforward map 
\begin{align}\label{equ psf map}i_*:P^t_{n}(S,\beta)\times T\to P^t_{n}(X,\beta), \end{align}
\begin{align*}(\oO_S\stackrel{s}{\to} F,\,t)\mapsto (\oO_X\twoheadrightarrow i_{t*}\oO_{S}\stackrel{i_{t*}s}{\to} i_{t*}F), \end{align*}
where $P^t_n(S,\beta)$ is the moduli space of $Z_t$-stable pairs $(F,s)$ on $S$ with $[F]=\beta$ and $\chi(F)=n$.
Without causing confusions, we use the same notation $t$ for both the stability parameter and closed points in $T$. 

We restrict to the following setting. 
\begin{set}\label{setting} 
We consider the case when the following conditions are satisfied:
\begin{enumerate}
\item The map \eqref{equ psf map} is an isomorphism and $P^t_{n}(S,\beta)$ is smooth of dimension $\beta^2+n+1$.
\item There is a well-defined forgetful map 
$$f: P^t_{n}(S,\beta)\to M_n(S,\beta), \quad (\oO_S\to F)\mapsto F, $$
to the coarse moduli scheme $M_n(S,\beta)$ of one dimensional stable sheaves $F$ on $S$ with $[F]=\beta$ and $\chi(F)=n$.
\end{enumerate}
\end{set}
\begin{prop}\label{prop on smoothness}
Setting \ref{setting} is satisfied when $\beta$ is irreducible. 
\end{prop}
\begin{proof}
When $\beta$ is irreducible, $P^t_{n}(X,\beta)$ is independent of the choice of $t>\frac{n}{\omega\cdot \beta}$ \cite[Prop.~1.12]{CT1},
so we can set $t\to \infty$ and work with PT stability. 
The isomorphism follows from similar argument as \cite[Prop.~3.11]{CMT2}.
The key point is that for any such $Z_t$-stable pair $(F,s)$, $F$ is stable and therefore scheme theoretically supported on $S\times \{t\}$ for some 
$t\in T$ (\cite[Lem.~2.2]{CMT1}). 
The smoothness of $P^t_{n}(S,\beta)$ follows from \cite{KY}, \cite[Prop.~C.2]{PT2}.
\end{proof}

\subsection{Virtual classes}
We determine the virtual class of $P^t_{n}(X,\beta)$ in Setting \ref{setting}. Firstly recall:
\begin{defi}$($\cite[Ex.~16.52,~pp.~410]{Sw}, \cite[Lem.~5]{EG}$)$
Let $E$ be a $\mathrm{SO}(2n,\mathbb{C})$-bundle with a non-degenerate symmetric bilinear form $Q$ on a connected scheme $M$. 
Denote $E_+$ to be its positive real form\,\footnote{This means a real half dimensional subbundle such that $Q$ is real and positive definite on it. By
homotopy equivalence $\mathrm{SO}(m,\mathbb{C})\sim \mathrm{SO}(m,\mathbb{R})$, it exists and is unique up to isomorphisms.}.
The half Euler class of $(E,Q)$ is 
$$e^{\frac{1}{2}}(E,Q):=\pm\,e(E_+)\in H^{2n}(M,\mathbb{Z}), $$
where the sign depends on the choice of orientation of $E_+$. 
\end{defi}

\begin{defi}$($\cite{EG},~\cite[Def.~8.7]{KiP}$)$
Let $E$ be a $\mathrm{SO}(2n,\mathbb{C})$-bundle with a non-degenerate symmetric bilinear form $Q$ on a connected scheme $M$. 
An isotropic cosection of $(E,Q)$ is a map 
$$\phi: E\to \oO_M, $$
such that the composition 
$$\phi\circ \phi^{\vee}: \oO_M\to E^{\vee}\stackrel{Q}{\cong} E \to \oO_M$$
is zero. If $\phi$ is furthermore surjective, we define the (reduced) half Euler class: 
$$e_{\mathrm{red}}^{\frac{1}{2}}(E,Q):=e^{\frac{1}{2}}\left((\phi^{\vee}\oO_M)^{\perp}/(\phi^{\vee}\oO_M),\bar{Q}\right)\in H^{2n-2}(M,\mathbb{Z}), $$
as the half Euler class of the isotropic reduction.
Here $\bar{Q}$ denotes the induced non-degenerate symmetric bilinear form on $(\phi^{\vee}\oO_M)^{\perp}/(\phi^{\vee}\oO_M)$. 
\end{defi}
We show reduced half Euler classes are independent of the choice of surjective isotropic cosection. 
\begin{lem}$($\cite[Lem.~5.5]{COT1}$)$\label{lem on indep of cosec}
Let $E$ be a $\mathrm{SO}(2n,\mathbb{C})$-bundle with a non-degenerate symmetric bilinear form $Q$ on a connected scheme $M$ and 
$$\phi: E\to \oO_M $$
be a surjective isotropic cosection.
Then we can write the positive real form $E_+$ of $E$ as 
$$E_+=\eE_+\oplus \underline{\mathbb{R}}^2$$
such that 
$$e_{\mathrm{red}}^{\frac{1}{2}}(E,Q)=\pm\,e(\eE_+). $$
Moreover, it is independent of the choice of surjective cosection. 

In particular, when $E=\oO^{\oplus2} \oplus V$ such that $Q=\begin{pmatrix}
0 & 1 \\
1 & 0
\end{pmatrix} \oplus Q|_{V}$, we have  
$$e_{\mathrm{red}}^{\frac{1}{2}}(E,Q)=\pm\,e^{\frac{1}{2}}(V,Q|_{V}).$$
\end{lem}
Recall a $\mathrm{Sp}(2r,\mathbb{C})$-bundle (or symplectic vector bundle) is a complex vector bundle of rank $2r$ 
with a non-degenerate anti-symmetric bilinear form.
One class of quadratic vector bundles is given by tensor product of two symplectic vector bundles $V_1, V_2$. 
Their half Euler classes
can be computed using Chern classes of $V_1,V_2$. 
For our purpose, we restrict to the following case.
\begin{lem}$($\cite[Lem.~5.6]{COT1}$)$\label{lem on compu of half euler class}
Let $(V_1,\omega_1)$, $(V_2,\omega_2)$ be a $\mathrm{Sp}(2r,\mathbb{C})$ $($resp.~$\mathrm{Sp}(2,\mathbb{C})$-bundle$)$ on a connected scheme $M$. 
Then
$$(V_1\otimes V_2,\omega_1\otimes \omega_2)$$ 
defines a $SO(4r,\mathbb{C})$-bundle whose half Euler class satisfies 
$$e^{\frac{1}{2}}(V_1\otimes V_2,\omega_1\otimes \omega_2)=\pm\,\big(e(V_1)-c_{2r-2}(V_1)\cdot e(V_2)\big). $$
\end{lem}
We determine the (reduced) virtual class of $P^t_{n}(X,\beta)$. 
\begin{thm}\label{thm on vir clas}
In Setting \ref{setting}, for certain choice of orientation, we have 
\begin{equation}\label{vir class StimesT}
[P^t_{n}(X,\beta)]^{\mathrm{vir}}=
\left([P^t_{n}(S,\beta)]\cap f^*e(T_{M_n(S,\beta)})\right)\times[T]-e(T)\left([P^t_{n}(S,\beta)]\cap f^*c_{\beta^2}(T_{M_n(S,\beta)})\right), 
\end{equation}
where $f: P^t_{n}(S,\beta)\to M_{n}(S,\beta)$ is the map as in Setting \ref{setting}. 
\end{thm}
\begin{proof}
The proof is similar as \cite[Prop.~4.7]{CMT2}. 
Under the isomorphism \eqref{equ psf map}: 
\begin{align*}P^t_{n}(S,\beta)\times T\cong P^t_{n}(X,\beta), \end{align*}
the universal stable pair $\mathbb{I}_X=(\oO\to \mathbb{F}_X)$ of $P^t_{n}(X,\beta)$ satisfies 
\begin{align}\label{equ of univ sheaf on prod}\mathbb{F}_X=\mathbb{F}_S\boxtimes \oO_{\Delta_T}, \end{align}
where $\mathbb{I}_S=(\oO\to \mathbb{F}_S)$ is the universal stable pair of $P^t_{n}(S,\beta)$ and $\Delta_T$ denotes the diagonal.

As in \cite[Eqn.~(29)]{CMT2}, we have a distinguished triangle
\begin{align}\label{equ on dist tri1}\dR\hH om_{\pi_{P_X}}(\mathbb{I}_X,\mathbb{F}_X)
\to \dR\hH om_{\pi_{P_X}}(\mathbb{I}_X,\mathbb{I}_X)_0[1]\to \dR\hH om_{\pi_{P_X}}(\mathbb{F}_X,\oO)[2],   \end{align}
where $\pi_{P_X}: P^t_{n}(X,\beta)\times X\to P^t_{n}(X,\beta)$ is the projection. 

From stable pair $\mathbb{I}_X=(\oO\to \mathbb{F}_X)$ and Eqn.~\eqref{equ of univ sheaf on prod}, we get a distinguished triangle
\begin{align}\label{equ on dist tri2}\dR\hH om_{\pi_{P_X}}(\mathbb{F}_S\boxtimes \oO_{\Delta_T},\mathbb{F}_S\boxtimes \oO_{\Delta_T})
\to \dR\hH om_{\pi_{P_X}}(\oO,\mathbb{F}_S\boxtimes \oO_{\Delta_T})\to \dR\hH om_{\pi_{P_X}}(\mathbb{I}_X,\mathbb{F}_X).  \end{align}
By adjunction, we get an isomorphism 
\begin{align}\label{equ on iso of rhom}\dR\hH om_{\pi_{P_X}}(\mathbb{F}_S\boxtimes \oO_{\Delta_T},\mathbb{F}_S\boxtimes \oO_{\Delta_T})
\cong \dR\hH om_{\pi_{P_S}}(\mathbb{F}_S,\mathbb{F}_S)\boxtimes \wedge^iT_T[-i],
\end{align}
where $\pi_{P_S}\colon P^t_{n}(S,\beta)\times S\to P^t_{n}(S,\beta)$ is the projection. 

Combining \eqref{equ on dist tri2} and \eqref{equ on iso of rhom}, we obtain 
\begin{align}\label{equ on iso on rhomIf}
\dR\hH om_{\pi_{P_X}}(\mathbb{I}_X,\mathbb{F}_X)\cong \dR\hH om_{\pi_{P_S}}(\mathbb{I}_S,\mathbb{F}_S)\oplus 
\dR\hH om_{\pi_{P_S}}(\mathbb{F}_S,\mathbb{F}_S)\boxtimes (T_T\oplus \oO_T[-1]).
\end{align}
Combining \eqref{equ on dist tri1} and \eqref{equ on iso on rhomIf}, we obtain 
$$\eE xt^1_{\pi_{P_X}}(\mathbb{I}_X,\mathbb{I}_X)_0\cong \eE xt^0_{\pi_{P_S}}(\mathbb{I}_S,\mathbb{F}_S)\oplus T_T, $$
and an exact sequence 
\begin{align}\label{equ on exa seq}
0\to \eE xt^1_{\pi_{P_S}}(\mathbb{I}_S,\mathbb{F}_S)\oplus \eE xt^1_{\pi_{P_S}}(\mathbb{F}_S,\mathbb{F}_S)\boxtimes T_T\oplus 
\eE xt^0_{\pi_{P_S}}(\mathbb{F}_S,\mathbb{F}_S)\boxtimes\oO_T \to 
\eE xt^2_{\pi_{P_X}}(\mathbb{I}_X,\mathbb{I}_X)_0 \to \cdots. \end{align}
We claim that the second arrow above is an isomorphism, which can be done by a dimension counting.  
In fact, let $I=(\oO_S\to F)\in P^t_{n}(S,\beta)$, the cohomology of the distinguished triangle 
$$\RHom_S(F,F)\to \RHom_S(\oO_S,F)\to \RHom_S(I,F)$$
implies that $\Ext^i_S(I,F)=0$ for $i\geqslant 2$. 
In Setting \ref{setting}, we know $\ext^0_S(I,F)=\beta^2+n+1$, therefore 
$$\ext^1_S(I,F)=1. $$
As $F$ is stable, we have 
$$\ext^2_S(F,F)=\ext^0_S(F,F)=1, \quad \ext^1_S(F,F)=\beta^2+2. $$ 
So the rank of the second term of \eqref{equ on exa seq} is $2\beta^2+6$. One can easily check the rank of the third term in \eqref{equ on exa seq} 
is also $2\beta^2+6$ by Riemann-Roch formula and first condition of Setting \ref{setting}. 
To sum up, we get an isomorphism:
$$\eE xt^1_{\pi_{P_S}}(\mathbb{I}_S,\mathbb{F}_S)\oplus \eE xt^1_{\pi_{P_S}}(\mathbb{F}_S,\mathbb{F}_S)\boxtimes T_T\oplus 
\eE xt^0_{\pi_{P_S}}(\mathbb{F}_S,\mathbb{F}_S)\boxtimes\oO_T \cong 
\eE xt^2_{\pi_{P_X}}(\mathbb{I}_X,\mathbb{I}_X)_0. $$
As in \cite[Prop.~4.7]{CMT2}, one can show the decomposition in the LHS is also with respect to the Serre duality pairing on $\eE xt^2_{\pi_{P_X}}(\mathbb{I}_X,\mathbb{I}_X)_0$. The our claim follows from Lemmata \ref{lem on indep of cosec} and \ref{lem on compu of half euler class}.
\end{proof}

\subsection{Thom-Porteous formula}
As our insertion \eqref{equ on pri ins} depends only on the fundamental cycle of the universal sheaf, it is useful to know 
the pushforward of the virtual class \eqref{vir class StimesT} under the forgetful map. 
In this section, let $\beta \in H_2(S,\BZ)$ be an irreducible curve class, then $P^t_{n}(X,\beta)$ is independent of the choice of $t>\frac{n}{\omega\cdot \beta}$ \cite[Prop.~1.12]{CT1},
so we can set $t\to \infty$ and work with PT stability. 
Consider the forgetful  map
\[ f : P_n(S,\beta) \to M_n(S,\beta), \quad (\oO_S\to F)\mapsto F. \]
Recall that $P_n(S,\beta)$ is smooth of dimension $\beta^2 + n + 1$ and $M_n(S,\beta)$ is smooth of dimension $\beta^2 + 2$.
The image of $f$ in $M_n(S,\beta)$ is the locus
\begin{equation} \big\{ F \in M_n(S,\beta)\,|\,h^0(F) \geqslant 1 \big\}, \label{x} \end{equation}
where surjectivity follows since $\beta$ is irreducible and $F$ is pure, so any non-zero section $s \in H^0(S,F)$ must have zero-dimensional cokernel.
The expected dimension of sections is $\chi(F) = n$,
so the image is everything if $n=1$, a divisor if $n=0$ and a codimension $2$ cycle if $n=-1$.

Let $\mathbb{F}_S$ be a (twisted) universal sheaf on $M_n(S,\beta) \times S$.
If $n=1$ (or more generally, there exists a $K$-theory class pairing with $1$ with a sheaf parametrized by $M_n(S,\beta)$)
the twisted sheaf can be taken to be an actual sheaf.
For us here the difference will not matter, since we are only interested in the Chern character of the universal sheaf,
which can also be easily defined in the twisted case. We refer to \cite{Markman} for a discussion.
%

Let $\pi_{M}: M_n(S,\beta) \times S\to M_n(S,\beta)$ be the projection. 
We resolve the complex
$\dR \pi_{M\ast}(\mathbb{F}_S)$ by a $2$-term complex of vector bundles: 
$$\dR \pi_{M\ast}(\mathbb{F}_S)\cong (E_0 \xrightarrow{\sigma} E_1). $$
Then \eqref{x} is the \textit{degeneracy locus}
$$D_1(\sigma) = \big\{ x \in M_n(S,\beta)\,|\, \dim_{\mathbb{C}} \ker(\sigma(x)) \geqslant 1 \big\}. $$
By the \textit{Thom-Porteous formula} \cite[\S14.4]{Ful}
(see \cite[Prop.~1]{GT} for a modern treatment and observe that $P_n(S,\beta)$ is precisely what is called $\tilde{D}_1(\sigma)$ there), we get the following:
\begin{prop}\label{deg loci}
\begin{equation*}f_{\ast} [P_n(S,\beta)] =c_{1-n}(-\dR \pi_{M\ast}(\mathbb{F}_S))\cap [M_n(S,\beta)]. \end{equation*}
\end{prop}

We can calculate the right hand side above by Grothendieck-Riemann-Roch formula
\[ \ch( - \dR \pi_{M\ast}(\mathbb{F}_S) ) = - \pi_{M\ast}( \ch(\mathbb{F}_S)\cdot\pi_S^*\td(S )). \]
We obtain the following:
\begin{equation} \label{Pn expressions}
\begin{aligned}
f_{\ast} [P_1(S,\beta) ] & = 1, \\
f_{\ast} [P_0(S,\beta) ] & = -\pi_{M\ast}(\ch_3(\mathbb{F}_S))-2\pi_{M\ast}(\ch_1(\mathbb{F}_S)\pi_S^*\pt), \\
f_{\ast} [P_{-1}(S,\beta) ] & = \frac{1}{2}\left(c_{1}(-\dR
\pi_{M\ast}(\mathbb{F}_S))\right)^2+\pi_{M\ast}(\ch_4(\mathbb{F}_S))+2\pi_{M\ast}(\ch_2(\mathbb{F}_S)\pi_S^*\pt), 
\end{aligned}
\end{equation}
where we used Poincar\'e duality on the right to identify homology and cohomology and 
$\pt\in H^4(S)$ denotes the point class.  
A small calculation shows that the right hand side is indeed 
independent of the choice of universal family $\mathbb{F}$ (i.e. the formulae stay invariant under replacing $\mathbb{F}$ by $\mathbb{F} \otimes \pi_{M}^*\CL$
for $\CL\in \Pic(M_n(S,\beta))$).
This will be useful later on.

\subsection{Genus 0 in irreducible classes}
In this section, we prove Conjecture \ref{conj on DT4/GV} (1), (2) for irreducible curve classes. 
We first recall a result of Fujiki \cite{Fuji} and its generalization in \cite[Cor.~23.17]{GHJ}.
\begin{thm}\label{fujiki result}$($\cite{Fuji}, \cite[Cor.~23.17]{GHJ}$)$
Let $M$ be a hyperk\"ahler variety of dimension $2n$. Assume $\alpha\in H^{4j}(M,\mathbb{C})$ is of type $(2j, 2j)$ on all small deformation of $M$. Then there exists a constant $C(\alpha)\in\mathbb{C}$ depending only on $\alpha$ $($called Fujiki constant of $\alpha$$)$ such that
$$\int_{M}\alpha\cdot\beta^{2n-2j}=C({\alpha})\cdot q_{M}(\beta)^{n-j}, \quad \forall\,\, \beta\in H^2(M, \mathbb{C}), $$
where $q_M: H^2(M, \mathbb{C}) \to \mathbb{C}$ denotes the Beauville-Bogomolov-Fujiki form. 
\end{thm}
\begin{thm}\label{thm on g=0 conj on prod}
Let $X=S\times T$ and $\beta\in H_2(S,\mathbb{Z})\subseteq H_2(X,\mathbb{Z})$ be an irreducible curve class. Then Conjecture \ref{conj on DT4/GV} (1), (2)
hold.
\end{thm}
\begin{proof}
By Proposition \ref{prop on smoothness}, we have a forgetful map 
$$\bar{f}=(f,\id_T): P_{n}(X,\beta)=P_{n}(S,\beta)\times T\to M_{n}(S,\beta)\times T. $$
As our insertion \eqref{equ on pri ins} only involves fundamental cycle of the universal one dimensional sheaf, so it is the pullback $\bar{f}^*$
of a cohomology class from $M_{n}(S,\beta)\times T$. 
 
When $n>1$, we have 
$$\dim_{\mathbb{C}} P_{n}(S,\beta)=\beta^2+n+1>\beta^2+2=\dim_{\mathbb{C}}M_{n}(S,\beta).  $$ 
By Theorem \ref{thm on vir clas} and Proposition \ref{deg loci}, it is easy to see 
$$P_{n,\beta}(\gamma_1,\ldots,\gamma_l)=0, \quad n>1. $$
When $n=1$, we take insertion $\gamma,\gamma'\in H^4(X)$ for example (other cases follow from easier versions of the same argument).
Based on K\"unneth decomposition:
$$H^{4}(X)\cong (H^0(S)\otimes H^4(T))\oplus (H^2(S)\otimes H^2(T))\oplus (H^4(S)\otimes H^0(T)), $$
we write 
\begin{align*}
\gamma &=A_1\cdot 1\otimes \pt+D_1\otimes D_2+A_2\cdot \pt\otimes 1, \\
\gamma' &=A'_1\cdot 1\otimes \pt+D'_1\otimes D'_2+A'_2\cdot \pt\otimes 1. 
\end{align*}
By Eqn.~\eqref{equ of univ sheaf on prod}, the insertion becomes (see also \cite[Proof~of~Thm.~5.8]{COT1}):
\begin{align}\label{equ on pri ins on prod}\tau(\gamma)=(D_1\cdot\beta)\otimes D_2+A_2f^*\pi_{M*}(\pi_S^*\pt\cdot \ch_1(\mathbb{F}_S))\otimes 1, \end{align}
where $\pi_S$, $\pi_{M}$ are projections from $S\times M_n(S,\beta)$ to its factors. Hence
$$\tau(\gamma)\cdot\tau(\gamma')=(D_1\cdot\beta)\cdot (D_1'\cdot\beta)\otimes (D_2\cdot D_2')+A_2A_2'f^*\left(\pi_{M*}\left(\pi_S^*\pt\cdot \ch_1(\mathbb{F}_S)\right)\right)^2\otimes 1+\mathrm{others}, $$
where ``others'' lie in $H^2(P_{1}(S,\beta))\otimes H^2(T)$. 
By Theorem \ref{thm on vir clas}, we get 
\begin{align*}P_{1,\beta}(\gamma,\gamma')&=(D_1\cdot\beta)\, (D_1'\cdot\beta)\,\int_T(D_2\cdot D_2')\int_{P_{1}(S,\beta)}f^*e(T_{M_1(S,\beta)}) \\
& \quad -e(T)A_2A_2' \int_{P_{1}(S,\beta)}f^*\left(c_{\beta^2}(T_{M_1(S,\beta)})\cdot \pi_{M*}\left(\pi_S^*\pt\cdot \ch_1(\mathbb{F}_S)\right)^2\right) 
\\
&=(D_1\cdot\beta)\, (D_1'\cdot\beta)\,\int_T(D_2\cdot D_2')\int_{M_{1}(S,\beta)}e(T_{M_1(S,\beta)}) \\
& \quad -e(T)A_2A_2' \int_{M_{1}(S,\beta)}c_{\beta^2}(T_{M_1(S,\beta)})\cdot \pi_{M*}\left(\pi_S^*\pt\cdot \ch_1(\mathbb{F}_S)\right)^2 \\
&=(D_1\cdot\beta)\, (D_1'\cdot\beta)\,\int_T(D_2\cdot D_2')\,e(M_1(S,\beta)),
\end{align*}
where the second equality follows from Proposition \ref{deg loci}
and the last equality is proved using Fujiki formula (Theorem \ref{fujiki result}) 
and the evaluation
\[ q_M( \pi_{M*}\left(\pi_S^*\pt\cdot \ch_1(\mathbb{F}_S)\right) ) = 0, \]
(which follows for example from \cite[Proof~of~Thm.~5.8]{COT1}).
Conjecture \ref{conj on DT4/GV} (2) then reduces to \cite[Thm.~5.8]{COT1}.
\end{proof}

\subsection{Transport of integrals to the Hilbert schemes}\label{sect on trans}
To compute the stable pair theory on $P_n(S,\beta)$ for $n \leqslant 0$ we will need to
handle more complicated descendent integrals on $M_n(S,\beta)$.
As in \cite[\S 4.4]{COT1} which deals with the $n = 1$ case,
we use here the general framework of monodromy operators of Markman \cite{Markman} (see also \cite{OUniversality})
to reduce to the Hilbert schemes.

Consider the Mukai lattice, which is the lattice $\Lambda = H^{\ast}(S,\BZ)$ endowed with the Mukai pairing
\[ \langle x , y \rangle := - \int_S x^{\vee} y, \]
where, if we decompose an element $x \in \Lambda$ according to degree as $(r,D,n)$, we have written $x^{\vee} = (r,-D,n)$.
Given a sheaf or complex $E$ on $S$ the Mukai vector of $E$ is defined by
\[ v(E) = \sqrt{\td_S} \cdot \ch(E) \in \Lambda. \]
Let $M(v)$ be a proper smooth moduli space of stable sheaves on $S$ with Mukai vector $v \in \Lambda$ (where stability is with respect to some fixed polarization).
We assume that there exists a universal family $\BF$ on $M(v) \times S$.
If it does not exists, everything below can be made to work by working with the Chern character $\ch(\BF)$
of a quasi-universal family, see \cite{Markman} or \cite{OUniversality}.
Let $\pi_M, \pi_S$ be the projections to $M(v)$ and $S$.
One has the Mukai morphism 
$$\theta_{\BF} : \Lambda \to H^2(M(v)), $$
\[ \theta_{\BF}(x) = \left[ \pi_{M \ast}( \ch(\BF) \cdot \sqrt{\td_S} \cdot x^{\vee} ) \right]_{\deg = 2}, \]
where $[ - ]_{\deg = k}$ stands for extracting the degree $k$ component
and (as we will also do below) we have suppressed the pullback maps from the projection to $S$.
Define the universal class
\[ u_v = \exp\left( \frac{ \theta_{\BF}(v) }{\langle v,v \rangle} \right) \ch(\BF) \sqrt{\td_S}, \]
which is independent of the choice of universal family $\BF$.
For $x \in \Lambda$, consider the normalized descendents:
\[ B(x) := \pi_{M\ast}( u_v \cdot x^{\vee} ), \]
and let $B_k(x) = [ B(x) ]_{\deg=2k}$ its degree $2k$ component.

\begin{example}
For $v=(1,0,1-d)$, the moduli space becomes the punctual Hilbert scheme: $M(v) = S^{[d]}$. 
Then we have
\[ u_v = \exp\left( \frac{-\delta}{2d-2} \right) \ch( \CI_{\CZ} ) \sqrt{\td_S}, \]
where we let $\delta = \pi_{\ast} \ch_3( \CO_{\CZ} )$ (so that $-2 \delta$ is the class of the locus of non-reduced subschemes).

We define the standard descendents on the Hilbert scheme by
\[ \fG_d(\alpha) = \pi_{\ast}( \pi_S^{\ast}(\alpha) \ch_d(\CO_{\CZ}) ) \in H^{\ast}(S^{[d]}). \]
One obtains that
\begin{align*}
B_1(\pt) & = - \frac{\delta}{2d-2}, \\
B_2(\pt) & = \frac{1}{2} \frac{\delta^2}{(2d-2)^2} - \fG_2(\pt). 
\end{align*}
For a divisor $D \in H^2(S)$, one finds
\begin{align*}
B_1(D) & = \fG_2(D), \\
B_2(D) & = \fG_3(D) - \frac{\delta}{2d-2} \fG_2(D).
\end{align*}
\end{example}
Using the descendents $B_k(x)$, one allows to move between any two moduli spaces of stable sheaves on $S$
just by specifying a Mukai lattice isomorphism $g : \Lambda \to \Lambda$.
We give the details in the case of our interest, 
see \cite{Markman, OUniversality} for the general case.

As before let $\beta \in \Pic(S)$ be an irreducible effective class of square $\beta \cdot \beta = 2d-2$,
and let $n \in \BZ$. We want to connect the moduli spaces
\[ M_n(S,\beta) \,\, \rightsquigarrow \,\, S^{[d]}. \]
Let $\beta = e + (d-1) f$ where $e, f \in H^2(S,\BZ)$ span a hyperbolic lattice: $\BZ e \oplus \BZ f \cong \binom{0\ 1}{1\ 0}$.
We do not require $e,f$ to be effective here.
Define the isomorphism $g : \Lambda \to \Lambda$ by
\begin{align*}
1 \mapsto (0,-e, n ), \quad \pt \mapsto (0,f,0), \quad e \mapsto (1, -nf, 0), \quad f \mapsto (0,0,-1), \quad
g|_{ \{ 1,\pt, e, f \}^{\perp}} = \id.
\end{align*}
One sees that $g$ is an isometry of the Mukai lattice and that
\[ g_{\ast} ( 0, \beta, n) = (1,0, 1-d). \]
Then one has:
\begin{thm}$($Markman \cite{Markman}, reformulation as in \cite[Thm.~4]{OUniversality}$)$ \label{thm:Markman} For any $k_i \geqslant 0$ and $\alpha_i \in H^{\ast}(S)$ and any polynomial $P$,
\[
\int_{ M_n(S,\beta) } P( B_{k_i}(\alpha_i) , c_j( T_{M_n(S,\beta)} ) ) 
=
\int_{ S^{[d]} } P( B_{k_i}(g \alpha_i) , c_j( T_{S^{[d]}}  ) ). 
\]
\end{thm}

\subsection{Genus 1 in irreducible classes}
Recall the genus $1$ Gopakumar-Vafa invariants (Proposition \ref{prop on gw on prod}).
On the stable pair side, we have the following:
\begin{thm} \label{thm on g=1 conj on prod}
Let $\beta\in H_2(S,\mathbb{Z})\subseteq H_2(X,\mathbb{Z})$ be an irreducible curve class. Then for certain choice of orientation, we have  
\begin{align}\label{equ on P0 pro}P_{0,\beta}(\gamma)=e(T)\, N_{1}\left(\frac{\beta^2}{2}\right) \int_{S\times\pt}\gamma. \end{align}
In particular, Conjecture \ref{conj on DT4/GV} (3) holds in this case. 
\end{thm}
\begin{proof}
The strategy is as follows: First we write our stable pair invariants as integrals on the moduli spaces $M_0(S,\beta)$,
then express the integrand in terms of the classes $B_k(x)$ and then use
Markman's Theorem~\ref{thm:Markman} to reduce to an integral over the Hilbert scheme, which is known by the results of \cite{COT1}.

By Eqn.~\eqref{equ on pri ins on prod} and Theorem \ref{thm on vir clas} (choose the inverse orientation there), we have 
$$P_{0,\beta}(\gamma)=e(T)\, \int_{S\times\pt}\gamma\cdot\int_{P_{0}(S,\beta)}f^*\left(c_{\beta^2}(T_{M_0(S,\beta)})\cdot \pi_{M*}\left(\pi_S^*(\pt)\cdot \ch_1(\mathbb{F}_S)\right)\right).$$
Using Proposition \ref{deg loci}, we find
\[ P_{0,\beta}(\gamma)=e(T)\,
\int_{S\times\pt}\gamma\cdot\int_{M_{0}(S,\beta)}c_{\beta^2}(T_{M_0(S,\beta)})\cdot c_{1}(-\dR \pi_{M\ast}(\mathbb{F}_S))\cdot  
\pi_{M*}\left(\ch_1(\mathbb{F}_S)\cdot \pi_S^*(\pt)\right). \]
A calculation shows that we have
\[ B_1(\pt) = \pi_{\ast}( \ch_1(\BF_S)\,\pi_S^{\ast}(\pt) ). \]
Moreover, the expressions \eqref{Pn expressions} are
invariant under replacing $\ch(\BF_S)$ by $\ch(\BF_S) \exp( \ell )$ for any line bundle $\ell \in H^2(M_n(S,\beta))$.
Hence we can use $\ch(\BF_S') := \ch(\BF_S) \exp( \theta_{\BF_S}(v) / \langle v,v \rangle )$ which shows that
\begin{align*}
c_{1}(-\dR \pi_{M\ast}(\mathbb{F}_S))
& = -\pi_{M\ast}(\ch_3(\mathbb{F}_S'))-2\pi_{M\ast}(\ch_1(\mathbb{F}_S')\,\pi_S^*\pt) \\
& = - B_1\left( \sqrt{\td_S}^{-1} \right) - 2 B_1( \pt ) \\
& = - B_1( 1 + \pt ).
\end{align*}
We obtain that:
\begin{align*}
& \int_{M_{0}(S,\beta)}c_{2d-2}(T_{M_0(S,\beta)})\cdot c_{1}(-\dR \pi_{M\ast}(\mathbb{F}_S))\cdot  
\pi_{M*}\left(\ch_1(\mathbb{F}_S)\cdot \pi_S^*\pt\right) \\
= & - \int_{M_{0}(S,\beta)}c_{2d-2}(T_{M_0(S,\beta)}) B_1(1 + \pt) B_1(\pt) \\
= & - \int_{S^{[d]}}c_{2d-2}(T_{S^{[d]}}) B_1(-e+f) B_1(f) \\
= & - \int_{S^{[d]}}c_{2d-2}(T_{S^{[d]}}) \fG_2(-e+f) \fG_2(f) \\
= & - ((-e+f) \cdot f) C( c_{2d-2}(T_{S^{[d]}}) ) \\
= & N_1(d-1),
\end{align*}
where we used the $k=1$ case of \cite[Thm.~4.2]{COT1} in the last step. 
\end{proof}

\subsection{Genus 2 in irreducible classes}
Let $\beta_d \in H_2(S,\BZ) \subseteq H_2(X,\BZ)$ be an irreducible curve class of square $\beta_d^2 = 2d-2$.
Below, we use similar method to compute stable pair invariants $P_{-1,\beta_d}$ on $X$ for all $d$.
\begin{thm}\label{thm on P_-1} For certain choice of orientation, we have  
\begin{align*} \sum_{d \in\mathbb{Z}} P_{-1,\beta_d}\, q^d &= 
\left(\prod_{n \geqslant 1} (1-q^n)^{-24}\right) \left(24q \frac{d}{dq} G_2(q) - 24G_2(q) - 1 \right) \\
&= 72 q^2 + 1920 q^3 + 28440 q^4 + 305280 q^5 + 2639760 q^6 + 
 19450368 q^7 +  \cdots .
\end{align*}
In particular, Conjecture \ref{conj on DT4/GV} (4) holds in this case. 
\end{thm}


\begin{proof}
As in the genus $1$ case, by Theorem \ref{thm on vir clas} and Proposition \ref{deg loci} we have:
\[
P_{-1,\beta}=-e(T)\int_{M_{-1}(S,\beta)}c_{2d-2}(T_{M_{-1}(S,\beta)})\cdot c_{2}(-\dR \pi_{M\ast}(\mathbb{F}_S)).
\]
With the same discussion as before one gets:
\[
c_{2}(-\dR \pi_{M\ast}(\mathbb{F}_S)) = \frac{1}{2} B_1(1 + \pt)^2 + B_2(1 + \pt).
\]
Hence applying Markman's Theorem~\ref{thm:Markman}, we conclude
\begin{align*}
& \int_{M_{-1}(S,\beta)}c_{2d-2}(T_{M_{-1}(S,\beta)})\cdot c_{2}(-\dR \pi_{M\ast}(\mathbb{F}_S))  \\
= & \int_{M_{-1}(S,\beta)}c_{2d-2}(T_{M_{-1}(S,\beta)})\cdot \left( \frac{1}{2} B_1(1 + \pt)^2 + B_2(1 + \pt) \right) \\
= & \int_{S^{[d]}} c_{2d-2}( T_{S^{[d]}} ) \left( \frac{1}{2} B_1(-e+f - \pt)^2 + B_2(-e+f-\pt ) \right) \\
= & \int_{S^{[d]}} c_{2d-2}( T_{S^{[d]}} ) 
\frac{1}{2} \left[ \fG_2(-e+f) + \frac{\delta}{2d-2} \right]^2 \\
& \ + \int_{S^{[d]}} c_{2d-2}( T_{S^{[d]}} ) \left( \fG_3(-e+f) - \frac{\delta}{2d-2} \fG_2(-e+f) - \frac{1}{2} \frac{\delta^2}{(2d-2)^2} + \fG_2(\pt) \right) \\
= & \frac{1}{2} \left( (-e+f)^2 + \frac{ \delta \cdot \delta}{(2d-2)^2} \right) N_1(d-1)
- \frac{1}{2} \frac{ \delta \cdot \delta }{(2d-2)^2} N_1(d-1) + \int_{S^{[d]}} c_{2d-2}( T_{S^{[d]}} ) \fG_2(\pt) \\
= & - N_1(d-1) + \int_{S^{[d]}} c_{2d-2}( T_{S^{[d]}} ) \fG_2(\pt).
\end{align*}
Thus we conclude that
\begin{align*}\label{equ on P-1 pro}
P_{-1,\beta}=e(T) \left( N_1(d-1) - \int_{S^{[d]}} c_{2d-2}( T_{S^{[d]}} ) \fG_2(\pt) \right).
\end{align*}
The desired formula now follows by the evaluation given in \cite[Prop.~4.6]{COT1}:
\[ \sum_{d \geqslant 0} q^d \int_{S^{[d]}} c_{2d-2}( T_{S^{[d]}} ) \fG_2(\pt) = \prod_{n = 1}^{\infty} (1-q^n)^{-24} \left( G_2(q) + \frac{1}{24} \right).  \]
Finally, comparing with Proposition \ref{prop on gw on prod}, we are done. 
\end{proof}
\begin{rmk}\label{rmk on pri g=0}
By the global Torelli theorem, primitive curve classes on $K3$ surfaces can be deformed to irreducible curve classes. Combining Theorem \ref{thm on g=0 conj on prod}, Theorem \ref{thm on g=1 conj on prod}, Theorem \ref{thm on P_-1}, we know 
Conjecture \ref{conj on DT4/GV} also holds for primitive curve classes $\beta\in H_2(S)\subseteq H_2(X)$. \end{rmk}

\subsection{Genus 1:~multiple fiber classes of elliptic fibrations}
Let $X=E\times E\times T$ be the product two copies of an elliptic curve $E$ and a $K3$ surface $T$. It gives the trivial elliptic fibration 
\begin{align}\label{equ on trivial ell fib}\pi: X\to Y:=E\times T. \end{align} 
For multiple fiber classes of $\pi$ \eqref{equ on trivial ell fib}, we have the following closed evaluation:
\begin{thm}\label{thm1 on g=1 of multiple fiber}
Let $t > 0$ and $\gamma \in H^4(X)$. For certain choice of orientation, we have 
\begin{align}\label{equ on P0 multiple fiber}
\sum_{r\geqslant 0}P^t_{0,r[E]}(\gamma)\,q^r=24\,\left(\int_{E \times E \times \pt} \gamma\right)\cdot\sum_{m\geqslant 1} \sum_{n | m}n^2q^m. \end{align}
\end{thm}
\begin{proof}
By \cite[Prop.~5.3]{CT1}, we know $P^t_0(X,n[E])$ is independent of the choice of $t>0$, so we may set $t\to \infty$ and work with PT stability. 
As in \cite[Lem.~3.5]{CMT2}, there is an isomorphism 
$$\pi^*: \Hilb^n(Y) \cong P_0(X,n[E]), \quad I_Z\mapsto \pi^*I_Z. $$
For $I=\pi^*I_Z\in P_0(X,n[E])$, by projection formula and 
$$\dR\pi_*\oO_X\cong \oO_{Y}\oplus K_{Y}[-1], $$
we obtain  
$$\RHom_X(I,I)\cong \RHom_{Y}(I_Z,I_Z)\oplus \RHom_{Y}(I_Z,I_Z\otimes K_{Y})[-1]. $$
By taking the traceless part, we get 
\begin{align}\label{equ1 on ell fib}\Ext^2_X(I,I)_0&\cong \Ext^2_{Y}(I_Z,I_Z)_0\oplus \Ext^1_{Y}(I_Z,I_Z)_0 \\ \nonumber 
&\cong \Ext^2_{Y}(I_Z,I_Z)_0\oplus \Ext^2_{Y}(I_Z,I_Z)_0^{\vee}, 
\end{align}
where we use Serre duality in the second isomorphism.

Next we compare cosections on these obstruction spaces. By \cite[Lem.~9.4]{KiP}, we have a surjective isotropic cosection
\begin{align*}\phi_X: \Ext^2_X(I,I)_0 \stackrel{\mathrm{At}(I)}{\longrightarrow} \Ext^3_X(I,I\otimes T^*X)\stackrel{\mathrm{tr}}{\longrightarrow}
H^3(X,T^*X)\stackrel{H\sigma_X }{\longrightarrow} H^4(X,\wedge^4T^*X) \stackrel{\int}{\longrightarrow}\mathbb{C},  \end{align*}
where $\mathrm{At}(I)\in \Ext^1_X(I,I\otimes T^*X)$ denotes the Atiyah class of $I$, $H\in H^1(X,T^*X)$ is an ample divisor 
and $\sigma_X\in H^0(X,\wedge^2T^*X)$ is a holomorphic symplectic form of $X$. 
The cosection $\phi_X$ is isotropic 
from the proof of~\cite[Cor.~9.5]{KiP}, and 
the surjectivity of $\phi_X$
also 
follows from loc.~cit.~together with 
(in the notation of~\cite[Lem.~9.4]{KiP})
\begin{align*}
    \int_{X}\iota_{H\sigma_X}\sigma_X \cup \beta \neq 0,
\end{align*}
where $\beta$ is the Poincar\'{e} dual of the curve 
class of $I$,
and the above non-vanishing follows since
$H$ is ample and $\beta$ is effective. 

By the compatibility of Atiyah classes with map $\pi: X\to Y$ (ref.~\cite[Prop.~3.14]{BFl}), we have a commutative diagram 
\begin{align}\label{com diag on atiyah class}\xymatrix{
\Ext^2_X(I,I)_0  \ar[r]^{\mathrm{At}(I) \,\,\, \quad } & \Ext^3_X(I,I\otimes T^*X) \ar[r]^{ \quad \mathrm{tr}  } & H^3(X,T^*X) \ar[r]^{\mathrm{pr}\quad \,\,\,}  &
H^{1,1}(S)\otimes H^{0,2}(T) \\
\Ext^2_Y(I_Z,I_Z)_0 \ar[u]_{i} \ar[r]^{\mathrm{At}(I_Z) \,\,\, \quad  } & \Ext^3_Y(I_Z,I_Z\otimes T^*Y) \ar[u]_{  }  \ar[r]^{\quad \,\, \mathrm{tr} }  & H^3(Y,T^*Y) \ar[r]^{\cong\quad \quad } & H^{1,1}(E)\otimes H^{0,2}(T) \ar[u]_{\pi^*},  } \end{align}
where $i$ is the embedding in \eqref{equ1 on ell fib}, $tr$ denotes the trace map and 
$pr$ is the projection with respect to K\"unneth decomposition.
We define a cosection 
\begin{align*}\phi_Y: \Ext^2_Y(I_Z,I_Z)_0 \stackrel{\mathrm{At}(I_Z)}{\longrightarrow} \Ext^3_Y(I_Z,I_Z\otimes T^*Y)\stackrel{\mathrm{tr}}{\longrightarrow}
H^3(Y,T^*Y)\cong H^{1,1}(E)\otimes H^{0,2}(T)\stackrel{\epsilon}{\longrightarrow} \mathbb{C},  \end{align*}
$$\mathrm{where} \quad \epsilon(\alpha)=\int_{X} H \sigma_X\cdot\pi^*\alpha, \quad \alpha\in H^{1,1}(E)\otimes H^{0,2}(T). $$
It is easy to see $\phi_Y$ is a positive multiple of the standard cosection of $\Hilb^n(Y)$ (see~e.g.~\cite[Eqn.~(6)]{O2}), 
hence its reduced virtual class keeps the same.

By diagram \eqref{com diag on atiyah class}, we have a commutative diagram: 
\begin{align*}\xymatrix{
\Ext^2_X(I,I)_0  \ar[r]^{\quad \quad \phi_X} & \mathbb{C} \\
\Ext^2_Y(I_Z,I_Z)_0 \ar[u]_{i} \ar[ur]_{\quad \phi_Y}. &  } \end{align*}
We claim that $\Ker(\phi_Y)$ is a maximal isotropic subspace of 
$\Ker(\phi_X)/\mathrm{Im}(\phi^{\vee}_X)$. In fact, by taking dual, we have a commutative diagram 
\begin{align*}\xymatrix{
\mathbb{C} \ar[d]^{=} \ar[r]^{ \phi^{\vee}_X \quad \quad \,\, }  & \Ext^2_X(I,I)_0^{\vee}  \ar[d]^{i^{\vee}} \ar[r]^{Q_{\mathrm{Serre}} \quad\quad\quad\quad\quad\quad\quad\quad}_{\cong \quad\quad\quad\quad\quad\quad\quad\quad } & \Ext^2_X(I,I)_0\cong  \Ext^2_Y(I_Z,I_Z)_0\oplus 
\Ext^2_Y(I_Z,I_Z)_0^{\vee} \ar[dl]^{\pi_2}  \\
\mathbb{C}\ar[r]^{ \phi^{\vee}_Y \quad \quad \quad } & \Ext^2_Y(I_Z,I_Z)_0^{\vee}.   &  } \end{align*}
Since $\phi_Y$ is surjective, so   $\phi^{\vee}_Y$ is injective, therefore  
$$\mathrm{Im}(\phi^{\vee}_X)\cap\Ker(\phi_Y)\subseteq \mathrm{Im}(\phi^{\vee}_X)\cap \Ext^2_Y(I_Z,I_Z)_0=0,  $$
and $\Ker(\phi_Y)$ defines a subspace of $\Ker(\phi_X)/\mathrm{Im}(\phi^{\vee}_X)$. 
This is a maximal isotropic subspace as $i: \Ext^2_Y(I_Z,I_Z)_0\to\Ext^2_X(I,I)_0$ is so. 

The above construction works in family and therefore we have 
$$[P_0(X,n[E])]^{\vir}=[\Hilb^n(Y)]^{\vir}\in A_1(P_0(X,n[E])), $$
for certain choice of orientation in the LHS. Consider a commutative diagram:
\begin{align*} \xymatrix{
X \ar[d]_{\pi} & X\times P_0(X,n[E]) \ar[d]_{\bar{\pi}=(\pi,(\pi^*)^{-1})}  \ar[l]_{\pi_X \quad\quad} \ar[r]^{\quad \pi_P }    & P_0(X,n[E]) \ar[d]_{(\pi^*)^{-1}}^{\cong}   \\
Y & Y\times \Hilb^n(Y) \ar[l]_{\pi_{Y} \quad\quad }    \ar[r]^{\,\, \pi_M }   & \Hilb^n(Y),  } \quad \quad
\end{align*}
and denote $\mathcal{Z}\hookrightarrow Y\times \Hilb^n(Y)$ to be the universal 0-dimensional subscheme. Then 
\begin{align*}
P_{0,n[E]}(\gamma)&=\int_{[P_0(X,n[E])]^{\vir}}\pi_{P*}(\pi_X^*\gamma\cdot \bar{\pi}^*\ch_3(\oO_{\mathcal{Z}})) \\
&=\int_{[\Hilb^n(Y)]^{\vir}}\pi_{M*}\bar{\pi}_*(\pi_X^*\gamma\cdot \bar{\pi}^*\ch_3(\oO_{\mathcal{Z}})) \\
&=\int_{[\Hilb^n(Y)]^{\vir}} \pi_{M*}(\ch_3(\oO_{\mathcal{Z}})\cdot\bar{\pi}_*(\pi_X^*\gamma) ) \\
&=\int_{[\Hilb^n(Y)]^{\vir}} \pi_{M*}(\ch_3(\oO_{\mathcal{Z}})\cdot\pi_{Y}^*\pi_*\gamma ).
\end{align*}
The statement now follows from Proposition~\ref{prop:K3xE calculation} below.
\end{proof}

\begin{prop}
\label{prop:K3xE calculation}
Let $\omega \in H^2(E,\BZ)$ be the class of point and $D \in H^2(T,\BQ)$ any class. Then for any $n \geqslant 1$ we have:
\begin{align*}
\int_{[ \Hilb^n(E \times T) ]^{\text{vir}} } \pi_{M \ast}\big( \ch_3(\CO_{\CZ}) \pi_Y^{\ast}( \omega \otimes 1 ) \big) & = (-1)^{n+1} e(T)  \sum_{d|n} d^2, \\
\int_{[ \Hilb^n(E \times T) ]^{\text{vir}} } \pi_{M \ast}\big( \ch_3(\CO_{\CZ}) \pi_Y^{\ast}( 1 \otimes D ) \big) & = 0.
\end{align*}
\end{prop}
\begin{proof}
Write $\Hilb = \Hilb^n(T \times E)$ and consider the diagram
\[
\begin{tikzcd}
T \times E & \Hilb \times T \times E \ar[swap]{l}{\pi_{T \times E}} \ar{r}{\pi_M} \ar{d}{\tilde{p}} & \Hilb \ar{d}{p} \\
& \frac{ \Hilb \times T \times E }{E} \ar{r}{\pi_{M/E}} & \Hilb/E,
\end{tikzcd}
\]
where the quotient by $E$ is taken in the stacky sense.
The universal subscheme $\CZ \subset \Hilb \times T \times E$ has a natural $E$-linearization and hence arises from the pullback of
a subscheme $\CZ/E \subset (\Hilb \times T \times E)/E$. Moreover, as in \cite{O2},
there exists a natural (0-dimensional) virtual class $[ \Hilb/E ]^{\text{vir}}$ such that
\[ [ \Hilb ]^{\text{vir}} = p^{\ast} [\Hilb/E]^{\text{vir}}. \]
Since the virtual class of $\Hilb/E$ arises from a symmetric obstruction theory (on an \'etale cover of $\Hilb/E$), its degree 
can be computed by an Behrend weighted Euler characteristic \cite{B}:
\[ \int_{ [\Hilb/E]^{\text{vir}} } 1 = e\left( \Hilb/E, \nu \right). \]
We argue now as follows: Applying the pushpull formula and using $p \circ \pi_M = \pi_{M/E} \circ \tilde{p}$ we have
\begin{align*}
N_n & := \int_{[ \Hilb ]^{\text{vir}} } \pi_{M \ast}\big( \ch_3(\CO_{\CZ}) \pi_{T \times E}^{\ast}( \omega \otimes 1 ) \big) \\
& = \int_{[ \Hilb/E ]^{\text{vir}} } \pi_{M/E \ast} \tilde{p}_{\ast} \big( \ch_3(\CO_{\CZ}) \pi_{T \times E}^{\ast}( \omega \otimes 1 ) \big) \\
& = \int_{[ \Hilb/E ]^{\text{vir}} } \pi_{M/E \ast} \big( \ch_3(\CO_{\CZ/E}) \tilde{p}_{\ast}( \pi_{T \times E}^{\ast}( \omega \otimes 1 )) \big).
\end{align*}
Then by checking on fibers we have $\tilde{p}_{\ast}( \pi_{T \times E}^{\ast}( \omega \otimes 1 )) = 1$ 
as well as $\pi_{M/E \ast} \ch_3(\CO_{\CZ/E}) = n$. This implies that
\[ N_n = n \int_{ [\Hilb/E]^{\text{vir}} } 1 = n\cdot e\left( \Hilb/E, \nu \right) = 24 (-1)^{n-1} \sum_{d|n} d^2, \]
where for the last equality we have used \cite[Cor.~1]{OS}.

For the second integral we argue identically, but observe that we have
\[ \ch_3(\CO_{\CZ}) \pi_{T \times E}^{\ast}( 1 \otimes D ) = \tilde{p}^{\ast}( \ch_3(\CO_{\CZ/E}) \pi_T^{\ast}(D)), \]
so when pushing forward by $\tilde{p}$ the integral vanishes.
\end{proof}
Similarly, we can consider a nontrivial elliptic fibration: 
$$\bar{p}=(p,\id_T): X=S\times T\to \mathbb{P}^{1}\times T, $$
where $p: S\rightarrow\mathbb{P}^{1}$ is an elliptic $K3$ surface with a section $i$. Let $f$ be a generic fiber of $\bar{p}$.
\begin{thm}\label{thm2 on g=1 of multiple fiber}
Let $t>0$ and $\gamma \in H^4(X)$. Then for certain choice of orientation, we have 
\begin{align}\label{equ on P0 multiple fiber}
\sum_{r\geqslant 0}P^t_{0,r[f]}(\gamma)\,q^r=24\,\left(\int_{S \times \pt} \gamma\right)\cdot \sum_{m\geqslant 1}\sum_{n | m}n^2q^m. \end{align}
\end{thm}
\begin{proof}
The first proof is parallel to the proof of Theorem \ref{thm1 on g=1 of multiple fiber}. 
For the second part, we need to evaluate
\begin{equation} \int_{[ \Hilb^n(  \p^1 \times T) ]^{\text{vir}} } \pi_{M \ast}\big( \ch_3(\CO_{\CZ})\,\pi_Y^{\ast}( \omega \otimes 1 ) \big), \label{aaas} \end{equation}
where $\omega \in H^2(\p^1)$ is the class of a point.
We consider the degeneration of $T \times \p^1$ given by the product of $T$ with the degeneration of $\p^1$ into a chain of three $\p^1$'s.
By specializing the insertion $\omega$ to the middle factor, we
are reduced to an integral of the relative Hilbert schemes $\Hilb^n(T \times \p^1 / T_0 \cup T_{\infty} )$ with the same integrand.
But this integral is also the outcome of applying the degeneration formula to the integrals considered in Proposition~\ref{prop:K3xE calculation}
(under the degeneration of $E$ to a nodal $\p^1$).
Hence \eqref{aaas} is given by $(-1)^{n+1} e(T) \sum_{d|n} d^2$ as well.
For the analogue of the second integral in Proposition~\ref{prop:K3xE calculation}, the localization formula applied to the scaling action of $\BC^{\ast}$ on $\p^1$ shows that it vanishes.
\end{proof}
\begin{rmk}\label{rmk on impr}
On the product of two $K3$ surfaces, 
genus 1 Gopakumar-Vafa invariants in imprimitive classes are defined in \cite[Def.~A.1]{COT1}. In particular, for multiple fiber classes 
$\beta=r[f]$ above, by using \cite[Eqn.~(5.7)]{COT1}, we know $n_{1,r[f]}(\gamma)=0$ if $r>1$.  
\end{rmk}

\section{Hilbert schemes of two points on $K3$}

\subsection{Rational curves on exceptional locus}
Let $S$ be a $K3$ surface. 
Consider the Hilbert-Chow map 
$$\pi: \Hilb^2(S)\to \Sym^2(S) $$
to the symmetric product of $S$. Let $D$ be the exceptional divisor fitting into Cartesian diagram: 
\begin{align*} \xymatrix{
D \ar[d]_{\pi}  \ar[r]^{i \quad \,\,\, }    & \Hilb^2(S) \ar[d]^{\pi}   \\
S    \ar[r]^{\Delta \quad \,\,\, }   & \Sym^2(S),  } \quad \quad
\end{align*}
where $\Delta$ is the diagonal embedding. Note that $\pi: D\to S$ is a $\mathbb{P}^1$-bundle and any fiber of it has normal bundle 
$\oO_{\mathbb{P}^1}(-2,0,0)$. 
\begin{thm}\label{thm on hilbS}
When $t=\frac{n}{\omega\cdot\beta}+0^+$ $($i.e.~in JS chamber$)$,
Conjecture \ref{conj on DT4/GV} (1),~(2) hold for multiple fiber classes $\beta=r[\mathbb{P}^1]$ $($$r\geqslant 1$$)$ of $\pi$ as above. 
\end{thm} 
\begin{proof}
As in \cite[Lem.~6.4]{CT1}, \cite[Lem.~3.1]{CT3},
by Jordan-H\"older filtration, the JS moduli space $P_n^{\mathrm{JS}}(X,d[\mathbb{P}^1])$ is nonempty only if 
$$d\,|\,n, \,\,\, n>0. $$
so we may assume $n=m\cdot d$ for $m\in \mathbb{Z}_{\geqslant 1}$. Consider the map 
\begin{align}f: P_{md}^{\mathrm{JS}}(X,d[\mathbb{P}^1])\to \Sym^d(S),  \quad (F,s)\mapsto \pi_*[F]. \end{align}
As the insertion \eqref{equ on pri ins} only involves fundamental cycle of the universal one dimensional sheaf $\mathbb{F}$, we have 
$$[\mathbb{F}]=\bar{f}^*[\mathcal{Z}], $$
where $[\mathcal{Z}]\hookrightarrow \Sym^d(S)\times S$ is the class of incident subvariety and $\bar{f}=(f,\id_S)$.
Therefore  
$$\int_{[P_{md}^{\mathrm{JS}}(X,d[\mathbb{P}^1])]^{\rm{vir}}}\prod_{i=1}^l\tau(\gamma_i)
=\int_{[P_{md}^{\mathrm{JS}}(X,d[\mathbb{P}^1])]^{\rm{vir}}}f^*\Phi, $$
for some $\Phi\in H^{2(md+1)}(\Sym^d(S))$. When $m>1$, we have $md+1>2d$, therefore $\Phi=0$ and 
$$P_{md,d[\mathbb{P}^1]}^{\mathrm{JS}}(\gamma_1,\ldots,\gamma_l)=0, \quad \mathrm{if}\,\,m>1. $$
For $m=1$, by a Jordan-H\"older filtration argument as \cite[Lem.~3.1]{CT3} again, we have
\begin{align*}
    \Hilb^d(S) &\stackrel{\pi^*}{\cong} P_{d}^{\mathrm{JS}}(D,d[\mathbb{P}^1]) \cong P_{d}^{\mathrm{JS}}(X,d[\mathbb{P}^1]),  \\
I_Z &\mapsto \pi^*I_Z\mapsto (\oO_X\to i_*\pi^*\oO_Z). 
\end{align*}
For $I_X=(\oO_X\to i_*\pi^*\oO_Z)$, we write $I_D=(\oO_D\to\pi^*\oO_Z)$. As in \cite[Prop.~4.3]{CMT2}, \cite[Prop.~4.2]{CKM2}, 
we have a canonical isomorphism 
$$\Ext^0_D(I_D,\pi^*\oO_Z)\cong \Ext^1_X(I_X,I_X)_0, $$
and an inclusion of maximal isotropic subspace 
\begin{align}\label{equ1 on excep curve}\Ext^1_D(I_D,\pi^*\oO_Z)\hookrightarrow \Ext^2_X(I_X,I_X)_0. \end{align}
From distinguished triangle 
$$I_D\to \oO_D\to \pi^*\oO_Z, $$
we obtain a distinguished triangle 
$$\RHom_D(\pi^*\oO_Z,\pi^*\oO_Z)\to \RHom_D(\oO_D,\pi^*\oO_Z)\to\RHom_D(I_D,\pi^*\oO_Z). $$
By projection formula, we have 
$$\RHom_D(\pi^*\oO_Z,\pi^*\oO_Z)\cong \RHom_S(\oO_Z,\oO_Z), \,\,\, \RHom_D(\oO_D,\pi^*\oO_Z)\cong \RHom_S(\oO_S,\oO_Z). $$ 
Therefore we get an exact sequence 
\begin{align}\label{equ2 on excep curve}0=H^1(S,\oO_Z)\cong H^1(D,\pi^*\oO_Z)\to \Ext^1_D(I_D,\pi^*\oO_Z) \to \Ext^2_S(\oO_Z,\oO_Z)\to 0. \end{align}
By Serre duality, we have 
\begin{align}\label{equ3 on excep curve} \Ext^2_S(\oO_Z,\oO_Z)\cong  \Ext^0_S(\oO_Z,\oO_Z)^{\vee}\cong H^0(S,\oO_Z)^{\vee}. \end{align}
Combining Eqns.~\eqref{equ1 on excep curve}, \eqref{equ2 on excep curve}, \eqref{equ3 on excep curve}, we obtain a maximal isotropic subspace
$$H^0(S,\oO_Z)^{\vee}\hookrightarrow\Ext^2_X(I_X,I_X)_0. $$
Working in family, we see that the dual of tautological bundle $\oO_S^{[d]}$ on $\Hilb^d(S)$ is a maximal isotropic subbundle  
of the obstruction bundle of $P_{d}^{\mathrm{JS}}(X,d[\mathbb{P}^1])$. By Lemma \ref{lem on indep of cosec}, we obtain
$$[P_{d}^{\mathrm{JS}}(X,d[\mathbb{P}^1])]^{\vir}=[\Hilb^d(S)]\cap c_{d-1}\left(\oO_S^{[d]}\right), $$
for certain choice of orientation. As for insertions, consider the following diagram 
\begin{align*} \xymatrix{
S & D  \ar[l]_{\pi} \ar[r]^{i} &  X \\
S\times \Hilb^d(S) \ar[d]^{\pi_M} \ar[u]_{\pi_S}       & D\times \Hilb^d(S) \ar[u]_{\pi_D}  \ar[d]^{\pi_M}  \ar[l]_{\bar{\pi}=(\pi,\id)} \ar[r]^{\bar{i}=(i,\id)} & X\times \Hilb^d(S) \ar[d]^{\pi_M} \ar[u]_{\pi_X}    \\
\Hilb^d(S) &\Hilb^d(S)  & \Hilb^d(S),  }  
\end{align*}
let $\mathcal{Z}\hookrightarrow\Hilb^d(S)\times S$ denote the universal zero dimensional subscheme, then 
\begin{align*}
\tau(\gamma)&=\pi_{M*}\left(\pi_X^*\gamma\cdot\ch_3(\bar{i}_*\bar{\pi}^*\oO_{\mathcal{Z}})\right) \\
&=\pi_{M*}\left(\pi_X^*\gamma\cdot\bar{i}_*\bar{\pi}^*[\mathcal{Z}]\right) \\
&=\pi_{M*}\bar{i}_*\left(\bar{i}^*\pi_X^*\gamma\cdot\bar{\pi}^*[\mathcal{Z}]\right) \\
&=\pi_{M*}\bar{\pi}_{*}\left(\pi_D^*i^*\gamma\cdot\bar{\pi}^*[\mathcal{Z}]\right) \\
&=\pi_{M*}\left(\bar{\pi}_{*}\pi_D^*i^*\gamma\cdot [\mathcal{Z}]\right)\\
&=\pi_{M*}\left(\pi_S^*\pi_*i^*\gamma\cdot [\mathcal{Z}]\right)\in H^2(\Hilb^d(S)),
\end{align*}
which depends only on $[\mathcal{Z}]$ and hence it is a pullback from $\Sym^d(S)$ by the Hilbert-Chow map
$$\mathrm{HC} \colon \Hilb^d(S)\to \Sym^d(S). $$ 
To sum up, we have  
\begin{align}\label{equ on exc cur} P_{d,d[\mathbb{P}^1]}^{\mathrm{JS}}(\gamma_1,\ldots,\gamma_l)=\int_{\Hilb^d(S)}c_{d-1}\left(\oO_S^{[d]}\right)\cdot 
\prod_{i=1}^l\pi_{M*}\left(\pi_S^*\pi_*i^*\gamma_i\cdot [\mathcal{Z}]\right). \end{align}
When $d=1$, this reduces to \cite[Lem.~3.7]{COT1}. When $d>1$, we claim the above integral is zero. 
In fact, by \cite[Thm.~4.6]{Lehn}, 
we have the formula 
\begin{align*}
	\sum_{m\geqslant 0} c\left(\mathcal{O}_S^{[m]}\right)z^m =
	\mathrm{exp}\left(\sum_{m\geqslant 1} \frac{(-1)^{m-1}}{m} q_m(1) z^m   \right) \cdot 1
	\end{align*}
where $q_m(1)$ are linear maps (called Nakajima operators)
\begin{align*}
	q_m(1) \in \End(\mathbb{H}), \quad 
	\mathbb{H}=\bigoplus_{m\geqslant 0} H^{\ast}(\mathrm{Hilb}^m(S), \mathbb{Q}),
	\end{align*}
which is of bidegree $(m, 2m-2)$. 
By looking at the bidegree $(d, 2d-2)$-part, 
we have 
\begin{align*}
	c_{d-1}\left(\oO_S^{[d]}\right)=q_d(1)(1), \quad \mathrm{where} \,\, 1\in H^0(\Hilb^0(S)). 
	\end{align*}
By the definition of $q_d(1)$ in~\cite[Def.~2.3]{Lehn}, we have $q_d(1)(1)=\pi_{M\ast}[\qQ]$
where $\qQ$ is the cycle on 
$\Hilb^d(S) \times S$ supported on 
$(\xi, x)$ with $\mathrm{Supp}(\xi)=x$. 
Therefore 
we know $c_{d-1}\left(\oO_S^{[d]}\right)$ is supported on $\mathrm{HC}^{-1}(\Delta)$, where
$$\Delta=\big\{(x,\cdots,x)\in\Sym^d(S) \big\}\subseteq \Sym^d(S)$$ is the small diagonal. 
Our insertion is a pullback from $\Sym^d(S)$ and gives $(d+1)$-dimensional constraint on $\Sym^d(S)$. If $d>1$, 
$d+1>2=\dim_{\mathbb{C}}\Delta$, therefore the integral \eqref{equ on exc cur} is zero. 
\end{proof}

\subsection{Small degree curve classes on $X=T^*\mathbb{P}^2$}
When the $K3$ surface $S$ has a
$(-2)$-curve $C \subset S$, 
the Hilbert scheme $\Hilb^2(S)$ contains 
$\mathrm{Sym}^2(C) \subset \Hilb^2(S)$
as a Lagrangian subvariety.
For curve classes coming from $\mathrm{Sym}^2(C) \cong 
\mathbb{P}^2$, our invariants can be studied on the local model $X=T^*\mathbb{P}^2$. 

We have an identification of curve classes:
\[ H_2(X,\BZ) = H_2(\p^2, \BZ) = \BZ [ \ell ], \]
where $\ell \subset \p^2$ is a line.
Let $H \in H^2(T^{\ast} \p^2)$ be the pullback of hyperplane class
and identify $H_2(T^{\ast} \p^2, \BZ) \equiv \BZ$ by its degree against $H$.
Gopakumar-Vafa invariants are given as follows: 
\begin{prop}\emph{(\cite[Cor.~6.2]{COT1})}\label{cor on inte on local p2}
\begin{align*}
n_{0,d}(H^2,H^2)&=
\left\{\begin{array}{rcl} 1   &\mathrm{if} \,\, d=1, \\ 
 -1  &\mathrm{if} \,\,  d=2,  \\
  0    & \,\,  \mathrm{otherwise}. 
\end{array} \right.  \\
n_{1,1}(H^2)&=0,  \quad n_{2,1}=0. 
\end{align*}
\end{prop}
In the stable pair side, we compute invariants for small degree curve classes.
\begin{prop}\label{prop on tp2}
For certain choice of orientation, we have 
$$P_{1,1}(H^2,H^2)=1, \quad P_{1,2}(H^2,H^2)=-1, \quad P_{1,3}(H^2,H^2)=0,   $$
$$P_{0,1}(H^2)=P_{0,2}(H^2)=0, \quad P_{0,3}(H^2)=1,\quad P_{-1,1}=P_{-1,2}=P_{-1,3}=0.  $$
Moreover, $P^t_{n}(X,d)$ is independent of the choice of $t>n/d$ in the listed cases above. 

In particular, for $X=T^*\mathbb{P}^2$, we have
\begin{itemize}
\item Conjecture \ref{conj on DT4/GV} (2) holds when $d\leqslant 3$. 
\item Conjecture \ref{conj on DT4/GV} (3), (4) hold. 
\end{itemize}
\end{prop}
\begin{proof}
As noted in \cite[Proof~of~Lem.~6.3]{COT1}, we have a diagram 
\begin{align*} \xymatrix{
X=T^{\ast}\mathbb{P}^2  \ar@{^{(}->}[r]^{\,\, i\quad }   & \oO_{\mathbb{P}^2}(-1)^{\oplus 3} \ar[d]^{\pi}   \\
 & T,  } 
\end{align*} 
where $i$ is a closed imbedding (coming from the Euler sequence) and $\pi$ contracts $\mathbb{P}^2$ to a point in 
an affine scheme $T$. 
It is easy to see that any one dimensional closed subscheme $C\subset X$ with $[C]=d$ ($d=1,2$) satisfies $\chi(\oO_C)\geqslant 1$.
Therefore by \cite[Prop.~1,12]{CT1}, we know for $n=-1,0,1$ and $d\leqslant 3$, the moduli space  
$P_n^t(X,d)$ is independent of the choice of $t>n/d$. So we may take $t\to \infty$ and work with PT stability. 
Using similar analysis as \cite[Prop.~3.9]{CKM2}, we know all stable pairs $(\oO_X\stackrel{s}{\to} F)$ in the above cases are scheme theoretically supported on 
the zero section $\mathbb{P}^2\subset X$ and $F$ are stable. 
Then obviously $P_{-1}(X,d)=\emptyset$ if $d\leqslant 3$ and corresponding invariants vanish. 

When $n=1$, $d\leqslant 3$, the  isomorphism 
$$P_1(X,d)\cong M_1(X,d), \quad (\oO_X\to F)\mapsto F, $$
to the moduli space of one dimensional stable sheaves $F$ with $[F]=d[\ell]$ and $\chi(F)=1$ will reduce the computation
to the corresponding one on $M_1(X,d)$ \cite[Prop.~6.5]{COT1}.

When $d=1,2$, we have $P_0(X,d)=\emptyset$, so invariants are zero. 
For $d=3$, the support map 
$$P_0(X,3)\cong P_0(\mathbb{P}^2,3)\stackrel{\cong}{\to} |\oO_{\mathbb{P}^2}(3)|\cong\mathbb{P}^9, \quad F\mapsto \mathrm{supp}(F)    $$
is an isomorphism. The universal one dimensional sheaf satisfies $\mathbb{F}=\oO_{\mathcal{C}}$ for the universal $(1,3)$-divisor
$\mathcal{C}\hookrightarrow \mathbb{P}^9\times \mathbb{P}^2$.  
Let $\pi_M \colon
P_0(X,3)\times \mathbb{P}^2\to P_0(X,3)$ be the projection. Bott's formula \cite[pp.~4]{OSS} implies that 
\begin{align*}
    \dR\mathcal{H}om_{\pi_M}(\oO,\oO(-\mathcal{C})\boxtimes T^*\mathbb{P}^2) &\cong \oO_{\mathbb{P}^9}(-1)[-2]^{\oplus 8}, \\
\dR\mathcal{H}om_{\pi_M}(\oO,\oO(\mathcal{C})\boxtimes T^*\mathbb{P}^2) &\cong \oO_{\mathbb{P}^9}(-1)^{\oplus 8}, \\
\dR\mathcal{H}om_{\pi_M}(\oO,\oO\boxtimes T^*\mathbb{P}^2)&\cong \oO_{\mathbb{P}^9}[-1].
\end{align*}
Therefore, we have 
\begin{align*}
&\quad \, \dR\mathcal{H}om_{\pi_M}(\oO_{\mathcal{C}},\oO_{\mathcal{C}}\boxtimes T^*\mathbb{P}^2)[1]\\
&\cong \dR\mathcal{H}om_{\pi_M}(\oO(-\mathcal{C})\to \oO,(\oO(-\mathcal{C})\to \oO) \boxtimes T^*\mathbb{P}^2)[1] \\
&\cong \oO_{\mathbb{P}^9}(-1)^{\oplus 8}\oplus \oO_{\mathbb{P}^9}(1)^{\oplus 8} \oplus \oO_{\mathbb{P}^9}  \oplus \oO_{\mathbb{P}^9}.  
\end{align*}
By Grothendieck-Verdier duality, it is easy to see 
$$\oO_{\mathbb{P}^9}(-1)^{\oplus 8}\oplus \oO_{\mathbb{P}^9}$$ 
is a maximal isotropic subbundle of $\dR\mathcal{H}om_{\pi_M}(\oO_{\mathcal{C}},\oO_{\mathcal{C}}\boxtimes T^*\mathbb{P}^2)[1]$. 

Following the proof of Theorem \ref{thm on vir clas}, one can show the reduced virtual class of $P_0(X,3)$ can be calculated as the reduced half Euler class of 
the bundle $\dR\mathcal{H}om_{\pi_M}(\oO_{\mathcal{C}},\oO_{\mathcal{C}}\boxtimes T^*\mathbb{P}^2)[1]$.
Therefore 
$$[P_0(X,3)]^{\vir}=\pm e\left(\oO_{\mathbb{P}^9}(-1)^{\oplus 8}\right)\cap [\mathbb{P}^9] \in H_2(\mathbb{P}^9). $$
Let $h\in H^2(\mathbb{P}^9)$ denote the hyperplane class. It is straightforward to check 
$$\tau_0(H^2)=[h]. $$ 
By integration again the virtual class, we have the desired result.
\end{proof}

 \appendix
\section{A conjectural virtual pushforward formula}\label{sect on app}
Let $\beta\in H_2(X,\mathbb{Z})$ be an irreducible curve class on a holomorphic symplectic 4-fold $X$. 
There is a well-defined forgetful map 
\begin{equation}f \colon P_{n}(X,\beta)\to M_{n}(X,\beta), \quad (\oO_X\to F)\mapsto [F], \nonumber \end{equation}
to the coarse moduli scheme of one dimensional stable sheaves $F$ with $[F]=\beta$, $\chi(F)=n$.  

Motivated by the Thom-Porteous formula (Proposition \ref{deg loci}), we conjecture the following:  
\begin{conj}\label{conj on gen virt push for}
In the above setting, there exists a choice of orientation such that 
$$f_*[P_{n}(X,\beta)]^{\vir}=c_{1-n}(-\dR \pi_{M\ast}(\mathbb{F}))\cap [M_{n}(X,\beta)]^{\vir}, $$
where $\pi_{M}: M_{n}(X,\beta) \times X\to M_{n}(X,\beta)$ is the projection and $\mathbb{F}$ is a universal sheaf  (if exists). 
\end{conj}
\begin{rmk}
This should be proved by adapting Park's beautiful work on virtual pullback \cite{Park} to the cosection localized version.
\end{rmk}
We can rewrite the degree of $[P_{-1}(X,\beta)]^{\vir}$ as a descendent integral on $[M_{1}(X,\beta)]^{\vir}$. 
Let $\mathbb{F}_{\mathrm{norm}}$ be the normalized universal sheaf on $M_{1}(X,\beta)\times X$, i.e. $\det(\dR\pi_{M*}\mathbb{F}_{\mathrm{norm}})\cong \oO_{M_{1}(X,\beta)}$. 
\begin{prop}\label{prop on appe}
Assume Conjecture \ref{conj on gen virt push for}. For certain choice of orientation, we have 
$$P_{-1,\beta}=-\int_{[M_{1}(X,\beta)]^{\vir}}\pi_{M*}\left(\ch_6(\mathbb{F}_{\mathrm{norm}})\right)
-\frac{1}{12}\int_{[M_{1}(X,\beta)]^{\vir}}\pi_{M*}\left(\ch_4(\mathbb{F}_{\mathrm{norm}})\,\pi_X^*(c_2(X))\right). $$
Using notations from \cite[\S 2.1]{COT1}, this is written as 
\begin{equation} \label{DTPT}
P_{-1,\beta}=-\langle\tau_3(1) \rangle^{\DT_4}_{\beta}
-\frac{1}{12}\langle\tau_1(c_2(X)) \rangle^{\DT_4}_{\beta}. 
\end{equation}
\end{prop}
\begin{proof}
The derived dual gives an isomorphism 
$$M_{-1}(X,\beta)\cong M_{1}(X,\beta), \quad F\mapsto F^{\vee}. $$
Then the computation is finished by applying the Grothendieck-Riemann-Roch formula.  
\end{proof}
\begin{rmk}
Based on Conjecture \ref{conj on DT4/GV} (4), this reproduces genus 2 Gopakumar-Vafa invariants of $X$ and 
therefore providing a sheaf theoretic approach to them using descendent integrals on moduli spaces of one dimensional
stable sheaves as \cite{CT2}. 
\end{rmk}
\begin{prop}\label{prop on prod of k3 app}
Conjecture \ref{conj on gen virt push for} holds on the product $X = S \times T$ of two $K3$ surfaces.
In particular, Eqn.~\eqref{DTPT} holds in this case. 
\end{prop}
\begin{proof}
Say $\beta\in H_2(S,\mathbb{Z})\subseteq H_2(X,\mathbb{Z})$, we have  
\begin{equation}f=f_S\times \mathrm{id}_T: P_{n}(X,\beta)\cong P_{n}(S,\beta)\times T\to M_{n}(X,\beta)\cong M_{n}(S,\beta)\times T, \nonumber \end{equation}
for forgetful map $f_S: P_{n}(S,\beta)\to M_{n}(S,\beta)$. 

By Theorem \ref{thm on vir clas} and \cite[Thm.~5.7]{COT1}, for certain choice of orientation, we have 
\begin{align}\label{equ on appe 1} 
[P_{n}(X,\beta)]^{\mathrm{vir}}&=
\left([P_{n}(S,\beta)]\cap f_S^*e(T_{M_n(S,\beta)})\right)\times[T]-e(T)\left([P_{n}(S,\beta)]\cap f_S^*c_{\beta^2}(T_{M_n(S,\beta)})\right), \\
[M_{n}(X,\beta)]^{\mathrm{vir}}&=
\left([M_{n}(S,\beta)]\cap e(T_{M_n(S,\beta)})\right)\times[T]-e(T)\left([M_{n}(S,\beta)]\cap c_{\beta^2}(T_{M_n(S,\beta)})\right).  \nonumber 
\end{align}
Also note that a universal sheaf $\mathbb{F}$ on $M_n(X,\beta)\times X$ (if exists) is of form 
$$\mathbb{F}=\mathbb{F}_S\boxtimes \oO_{\Delta_T}, $$
where $\mathbb{F}_S$ is a universal sheaf on $M_n(S,\beta)\times S$ and $\Delta_T$ is the diagonal of $T\times T$. So 
\begin{equation}\label{equ on appe 2} \dR\pi_{M*}\mathbb{F}=\dR\pi_{M_S*}\mathbb{F}_S, \end{equation}
where $\pi_{M_S}: M_n(S,\beta)\times S\to M_n(S,\beta)$ is the projection. 
Combining Eqns.~\eqref{equ on appe 1}, \eqref{equ on appe 2}, we are reduced to Proposition \ref{deg loci}.
\end{proof}

\providecommand{\bysame}{\leavevmode\hbox to3em{\hrulefill}\thinspace}
\providecommand{\MR}{\relax\ifhmode\unskip\space\fi MR }
\providecommand{\MRhref}[2]{%
 \href{http://www.ams.org/mathscinet-getitem?mr=#1}{#2}}
\providecommand{\href}[2]{#2}

\end{document}